\documentclass[12pt]{amsart}
\usepackage{amsmath}
\usepackage{amssymb}
\usepackage{amsthm}
\usepackage{amscd, newlfont, enumerate, latexsym, mathrsfs}
\usepackage{mathtools}
\usepackage[latin2]{inputenc}
\usepackage[cmtip,arrow]{xy}
\usepackage{pb-diagram, pb-xy}
\usepackage[active]{srcltx}
\usepackage{a4wide}

\usepackage{hyperref}

\usepackage{color}
\usepackage{xcolor}

 \newtheorem{theorem}{Theorem}[section]
 \newtheorem{corollary}[theorem]{Corollary}
 \newtheorem{lemma}[theorem]{Lemma}
 
 \newtheorem{definition}[theorem]{Definition}
 \newtheorem{remark}[theorem]{Remark}

\parskip=\bigskipamount

\newcommand{\Jord}{\mathop{\text{Jord}}}

\def\pc{\pi_{cusp}}
\def\Z{\mathbb Z}
\def\R{\mathbb R}
\def\d{\delta}
\def\r{\rtimes}
\def\t{\times}
\def\o{\otimes}
\def\e{\epsilon}

\def\h{\hookrightarrow}
\def\ra{\rightarrow}

\def\a{\alpha}
\def\b{\beta}

\def\vf{\varphi}

\def\s{\sigma}

\def\ep{\epsilon_\pi}

\def\BE{\begin{equation}}
\def\EE{\end{equation}}

\title{On Jacquet modules of representations of segment type}
\author{Ivan Mati\'{c} and Marko Tadi\'c}
\date{\today}

\begin{document}


\begin{abstract}
Let $G_{n}$ denote either the group $Sp(n, F)$ or $SO(2n+1, F)$ over a local non-archimedean field $F$. We study representations of segment type of group $G_{n}$, which play a fundamental role in the constructions of discrete series, and obtain a complete description of the Jacquet modules of these representations. Also, we provide an alternative way for determination of Jacquet modules of strongly positive discrete series and a description of top Jacquet modules of general discrete series.
\end{abstract}

\maketitle

{\renewcommand{\thefootnote}{}
\footnotetext[1]{I. Mati\'{c}: Department of Mathematics, University of Osijek, Osijek, Croatia, e-mail: imatic@mathos.hr}
\footnotetext[2]{M. Tadi\'{c}: Department of Mathematics, University of Zagreb, Zagreb, Croatia, e-mail: tadic@math.hr}
\footnotetext[3]{\textit{Mathematics Subject Classification:}
22E35 (primary), 22E50, 11F70 (secondary)}

\section{Introduction}

Let $F$ be a local non-archimedean field of characteristic different than two. Representations of reductive groups over $F$ that we  shall consider in this paper will be always smooth and admissible. We shall use  standard notation from the representation theory of general linear groups over $F$ introduced by Bernstein and Zelevinsky (see \cite{Z}).  Recall that Levi factors of maximal parabolic subgroups of general linear groups are direct products of two smaller general linear groups. This fact enables one to consider the representation parabolically induced by the tensor product $\pi_1\o\pi_2$ of two representations of general linear groups, which is denoted by
$$
\pi_1\times \pi_2.
$$
The parabolic induction that we consider in this paper will always be from the parabolic subgroups standard with respect to the subgroup of upper triangular matrices (the same will be the case for Jacquet modules). The Grothendieck group of the category of finite length representations of $GL(n,F)$ is denoted by $R_n$. The parabolic induction $\times$ defines in a natural way the structure of a commutative graded algebra with unit on $R=\oplus_{n\in\Z_{\geq 0}}R_n.$
The induced map from $R\o R$ will be denoted by $m$. (Sums of semi simplifications of) Jacquet modules with respect to maximal parabolic subgroups define mapping $m^* \colon R\ra R\o R$.  This gives $R$ the structure of a graded coalgebra. Moreover, $R$ is a Hopf algebra.

Denote
$$
\nu \colon GL(n,F)\ra \R^\t, \quad g \mapsto |\det(g)|_F
$$
where $|\ |_F$ denotes the normalized absolute value. A segment is a set of the form $\{\rho,\nu\rho, \nu^2\rho,$ $\dots,\nu^k\rho\}$,
where $\rho$ is an irreducible cuspidal representation of a general linear group. We denote this set shortly by $[\rho,\nu^k\rho]$.
To such a segment the unique irreducible subrepresentation of $\nu^k\rho\t\dots\t\rho$ is attached, which we denote by $\d([\rho,\nu^k\rho]).$
These are the essentially square integrable representations, and one gets all such representations in this way.

A very important  (and very simple) formula of Bernstein-Zelevinsky theory is
\begin{equation*} \label{prva}
m^*(\d([\rho,\nu^k\rho]))=\sum_{i=-1}^{k} \d([\nu^{i+1}\rho,\nu^k\rho]) \o \d([\rho,\nu^i\rho]),
\end{equation*}
which by the transitivity of Jacquet modules, describes all Jacquet modules of irreducible essentially square integrable representations of general linear groups.

One would also like to have such a formula to determine Jacquet modules of representations of classical groups. It is of particular interest to determine Jacquet modules of classes of representations of classical groups whose role in the admissible dual is as important as the role of essentially square integrable representations in the admissible dual of a general linear group. Besides being interesting in itself, such description would have applications in the theory of automorphic forms and in the classification of unitary duals.

In the present paper we are concerned with representations of segment type of symplectic and special odd-orthogonal groups over $p$-adic field $F$. This prominent class of representations, consisting of certain irreducible subquotients of representations induced by the tensor product of an essentially square integrable representation of a general linear group and a supercuspidal representation of a classical group, has been introduced by the second author in \cite{T-seg}. Such representations have also appeared as the basic ingredients in classifications of discrete series and tempered representations of classical groups (we refer the reader to \cite{Moe-T} and \cite{T-temp}). Representations of segment type can be viewed as irreducible subquotients of generalized principal series induced from representation having a supercuspidal classical-group part. We note that composition series of such representations have been obtained by Mui\'{c} in \cite{Mu-CSSP} (in fact, a more general class of generalized principal series, having a strongly positive representation on the classical-group part, has been studied there). In determination of the composition series of induced representation, the fundamental role is played by Jacquet modules of the initial representation. Thus, our results provide a starting point for investigation of representations induced by those of segment type.

We emphasize that in several cases our results provide complete description of Jacquet modules of certain non-tempered representations. Our description can be used to analyze asymptotics of matrix coefficients of such representations and, consequently, to determine some prominent members of the unitary dual.

In the case of generic reducibilities, representations of segment type are always tempered or discrete series representations. However, for general reducibilities, representations of segment type might also be non-tempered. The structural formula, which is a version of the Geometrical Lemma of Bernstein-Zelevinsky, together with certain properties of the representations of segment type obtained in \cite{T-seg}, enables us to use an inductive procedure which results in a complete description of Jacquet modules of $GL$-type and top Jacquet modules of such representations. These results, enhanced by the transitivity of Jacquet modules and some results regarding Jacquet modules of representations of general linear groups, allow us to determine Jacquet modules of representations of segment type with respect to all standard maximal parabolic subgroups. Since representations of segment type can appear
in three technically different composition series, we obtain a description of their Jacquet modules considering three technically different cases (but the general strategy in all the cases is the same).
However, we introduce a convention regarding irreducible constituents of considered composition series, which enables us to state our results uniformly.

An analogous problem to determine Jacquet modules has been studied for strongly positive representations by the first author (\cite{Ma-JMSP}), but it was mostly based on the fact that the Jacquet module of strongly positive discrete series has a representation of the same type on its classical-group part.

On the other hand, an approach similar to the one presented here has recently been used by the first author to provide a description of Jacquet modules with respect to maximal parabolic subgroups of certain families of discrete series which contain an irreducible essentially square integrable representation on the $GL$-part (\cite{Ma-JMDS1}). In that paper one starts with determination from Jacquet modules which are not of $GL$-type. Then, to deduce to which irreducible subquotient obtained Jacquet modules belong, one uses transitivity of Jacquet modules and representation theory of general linear groups.

Let us now describe the contents of the paper in more details. In the following section we introduce some notation which will be used throughout the paper, while in the third section we recall some important properties of representations of segment type and introduce a certain convention which will keep our results uniform. The next three sections are devoted to determination of Jacquet modules of the representations of segment type,
considering three technically different cases. Also, some results obtained in section four and five help us to shorten the proofs in the sixth section. In the last two sections we derive some interesting Jacquet modules of discrete series. Firstly we provide an alternative way to determine Jacquet modules of strongly positive discrete series and secondly we provide a description of top Jacquet modules of general discrete series.

For the convenience of the reader, we cite the main description of Jacquet modules here.

Representations of segment type are irreducible subqutients of  $\delta([\nu^{-c}\rho,\nu^{d}\rho])\r\s$, where $\rho$ is an irreducible cuspidal representations of
a general linear group and  $\sigma$ is an irreducible cuspidal representations of a classical group ($c,d \in \R$, $c+d\in  \mathbb{Z}_{\geq 0}$). Then directly
from previously mentioned formula for $m^*$ and \cite{T-Str} we get
\begin{equation*}
\mu^*\left(\delta([\nu^{-c}\rho,\nu^{d}\rho])\r\s\right)
 =
\sum_{i= -c-1}^{ d}
\sum_{j=i}^{ d}
\delta([\nu^{-i}\tilde\rho,\nu^{c}\tilde\rho])
 \times
\delta([\nu^{j+1} \rho,\nu^{d}\rho]) \otimes
\delta([\nu^{i+1} \rho,\nu^{j}\rho])\r\s
\end{equation*}
($\tilde \rho$ denotes the contragredient of $\rho$). The case which interest us is when $\delta([\nu^{-c}\rho,\nu^{d}\rho])\r\s$ reduces (square integrable subquotients can show up in this case  only).
Then we can take selfcontragredient $\rho$ and assume $d\in (1/2)\mathbb{Z}$ (only in this case we can have reducibility).
We shall consider the case $d-c\geq 0$ (changing signs of $c$ and $d$ simultaneously gives the same composition series).
We shall assume that there exists $\a\in(1/2)\Z_{\geq0}$ such that for $\b\geq 0$, $\nu^\b\rho\r\s$ reduces if and only if $\b=\a$. This always holds for
$\rho$ selfcontragredient (it is a very non-trivial fact which we shall not discuss here; we shall simply assume that it  holds for $\rho$ and $\s$).
Also, we assume $d-\a\in\Z$ (only then we can have reducibility).

The length of $\delta([\nu^{-c}\rho,\nu^{d}\rho])\r\s$ is at most three. This is a multiplicity one representation. It is reducible  if and only if $[-c,d]\cap\{-\a,\a\}\ne \emptyset$.
It has length three if and only if $\{-\a,\a\}\subseteq[-c,d]$ and $c\ne d$.

Below we shall define terms $\delta([\nu^{-c}\rho,\nu^{d}\rho]_+;\s)$, $\delta([\nu^{-c}\rho,\nu^{d}\rho]_-;\s)$ and $L_\a(\delta([\nu^{-c}\rho,\nu^{d}\rho]);\s)$. Each of them is either irreducible representation or zero. They satisfy
\begin{equation} \label{uuvodu}
\delta([\nu^{-c}\rho,\nu^{d}\rho])\rtimes\s=\delta([\nu^{-c}\rho,\nu^{d}\rho]_+;\s)+\delta([\nu^{-c}\rho,\nu^{d}\rho]_-;\s)+L_\a(\delta([\nu^{-c}\rho,\nu^{d}\rho]);\s)
\end{equation}
in the corresponding Grothendieck group.

Suppose first that $\delta([\nu^{-c}\rho,\nu^{d}\rho])\r\s$ is irreducible.  Then we take $\delta([\nu^{-c}\rho,\nu^{d}\rho]_-;\s)=0$. Furthermore, in this case we require
$\delta([\nu^{-c}\rho,\nu^{d}\rho]_+;\s)\ne0$ if and only if $[-c,d]\subseteq[-\a+1,\a-1]$. For irreducible $\delta([\nu^{-c}\rho,\nu^{d}\rho])\r\s$, this requirement and (\ref{uuvodu}) obviously determine $L_\a(\delta([\nu^{-c}\rho,\nu^{d}\rho]);\s)$.

Suppose now that $\delta([\nu^{-c}\rho,\nu^{d}\rho])\r\s$ reduces. If $c=d$, let $L_\a(\delta([\nu^{-c}\rho,\nu^{d}\rho]);\s)=0$. Otherwise, $L_\a(\delta([\nu^{-c}\rho,\nu^{d}\rho]);\s)$
will denote the Langlands quotient $L(\delta([\nu^{-c}\rho,\nu^{d}\rho]);\s)$ of $\delta([\nu^{-c}\rho,\nu^{d}\rho])\r\s$.
If $\a > 0$, then there is the unique irreducible subquotient of $\delta([\nu^{-c}\rho,\nu^{d}\rho])\r\s$ which has in its minimal standard Jacquet module
at least one irreducible subquotient whose all exponents are non-negative (for more details, we refer the reader to Sections 2 and 3). We denote such irreducible subquotient of $\delta([\nu^{-c}\rho,\nu^{d}\rho])\r\s$ by $\delta([\nu^{-c}\rho,\nu^{d}\rho]_+;\s)$.
If $\a = 0$, we write $\rho \rtimes \sigma$ as a sum of irreducible subrepresentations  $\tau_1 \oplus \tau_{-1}$ (we fix the choice of signs $\pm$, which is arbitrary and can be compatible with the one from \cite{Moe-Ex},
but this is not essential for our paper). Then there exists the unique irreducible subquotient of $\delta([\nu^{-c}\rho,\nu^{d}\rho]) \rtimes \s$ that contains an irreducible representation
of the form $\pi \otimes \tau_1$ in Jacquet module with respect to appropriate standard parabolic subgroup, and we denote it by $\delta([\nu^{-c}\rho,\nu^{d}\rho]_+;\s)$.
If $c = d$ or the length of $\delta([\nu^{-c}\rho,\nu^{d}\rho])\r\s$ is three, then this induced representation contains the unique irreducible subrepresentation different from
$\delta([\nu^{-c}\rho,\nu^{d}\rho]_+;\s)$ and we denote it by $\delta([\nu^{-c}\rho,\nu^{d}\rho]_-;\s)$. Otherwise, we take $\delta([\nu^{-c}\rho,\nu^{d}\rho]_-;\s)=0$.
We note that the representation $\delta([\nu^{-c}\rho,\nu^{d}\rho]_+;\s)$ is square integrable if and only if $c\ne d$, $\{-\a,\a\}\subseteq [-c,d]$ or $\a=-c$.
If $\delta([\nu^{-c}\rho,\nu^{d}\rho]_+;\s)$ is square integrable, then $\delta([\nu^{-c}\rho,\nu^{d}\rho]_-;\s)$ is also square integrable, if it is non-zero.
Furthermore, if $\delta([\nu^{-c}\rho,\nu^{d}\rho]_-;\s)$ is square integrable, then $\delta([\nu^{-c}\rho,\nu^{d}\rho]_+;\s)$ is square integrable.

We have the following equality:
 \begin{align*}
\mu^*\big(&\delta([\nu^{-c}\rho,\nu^{d}\rho]_\pm;\s)\big)  
\\
& = \sum_{i= -c -1}^{ d-1
}\sum_{j=i+1}^{d} \delta([\nu^{-i}\rho,\nu^{c}\rho])  \times \delta([\nu^{j+1}\rho,\nu^{d}\rho]) \otimes
\delta([\nu^{i+1} \rho,\nu^{j}\rho]_\pm;\s) +{} 
\\
& + \mkern-20mu\sum_{-c-1\le i\le c-1}\ \sum_{i+1\le j\le
c}\mkern-75mu\rule[-4.5ex]{0pt}{2ex}_{i+j < -1}\mkern25mu
\delta([\nu^{-i}\rho,\nu^{c}\rho])
 \times
\delta([\nu^{j+1} \rho,\nu^{d}\rho]) \otimes
L_\a(\delta([\nu^{i+1} \rho,\nu^{j}\rho]);\s) +
\\[-1.5ex]
&\mkern100mu+ \mkern-10mu\sum_{i=-c-1}^{\pm \a - 1}
\delta([\nu^{-i}\rho,\nu^{c}\rho]) \t \delta([\nu^{i+1}\rho,\nu^{d}\rho])\o\s.
\end{align*}
{\color {blue} In the  above formula we have corrected two typographical errors which exist in the published version of this paper. First,  the upper limit in the first sum of the second row  needs to be $d-1$
(instead of $c$, as it is in the published version).  Then, the limits of the first sum in the third row  are $-c-1\leq i\leq c-1$ (instead of $-c-1\leq i\leq c $\,; the index  $c$ does not give any contribution).
The same corrections are made to the corresponding formulas in Corollaries \ref{cor1}, \ref{korprvi} and \ref{korolar-drugi}.}

For $c < \a$ or $\a \leq c<d$, we have
\begin{align*}
\mu^*\big(&L(\delta([\nu^{-c}\rho,\nu^d\rho]);\sigma)\big) 
\\
& = \mkern-20mu\sum_{-c-1\le i\le d-1}\ \sum_{i+1\le j\le
d}\mkern-75mu\rule[-4.5ex]{0pt}{2ex}_{0\le i+j}\mkern25mu
L(\delta([\nu^{-i}\rho,\nu^c\rho]),\delta([\nu^{j+1}\rho,\nu^d\rho])\big) \otimes
L_\alpha(\delta([\nu^{i+1}\rho,\nu^j\rho]);\sigma)) +{}  
\\[-1.5ex]
&\mkern100mu+ \mkern-10mu\sum_{i=\a}^d
L(\delta([\nu^{-i}\rho,\nu^c\rho]),\delta([\nu^{i+1}\rho, \nu^d\rho])) \otimes
\sigma. 
\end{align*}
{\color {blue} Also, in the  above formula we have corrected a typographical error  existing in the published version:
 the limits in the first sum in the second row  are $-c-1\leq i\leq d-1$ (instead of $-c-1\leq i\leq d$\,; the index $d$ does not contribute 
in the formula). 
The same correction is made to the corresponding formulas in Corollaries \ref{cor1}, \ref{korprvi} and \ref{korolar-drugi}.}

The authors would like to thank the referee for reading the paper very carefully and helping us to improve the presentation style. Also, the author's thanks go to \v{S}ime Ungar for many useful suggestions and help with English language.

This work has been supported by Croatian Science Foundation under the project 9364.

\section{Notation}

We will first describe the groups that we consider.

Let $J_{n} = (\delta_{i, n+1-j})_{1 \leq i,j \leq n}$ denote the $n \times n$ matrix, where $\delta_{i, n+1-j}$ stands for the Kronecker symbol. For a square matrix $g$, we denote by $g^{t}$ its transposed matrix, and by $g^{\tau}$ its transposed matrix with respect to the second diagonal. In what follows, we shall fix one of the series of classical groups
\begin{equation*}
Sp(n, F) = \bigg \{ g \in GL(2n, F) : \left( \begin{array}{cc} 0 & -J_{n} \\ J_{n} & 0 \end{array} \right) g^{t}  \left( \begin{array}{cc} 0 & -J_{n} \\ J_{n} & 0 \end{array} \right) =  g^{-1}\bigg \},
\end{equation*}
or
\begin{equation*}
SO(2n+1, F) = \bigg \{ g \in GL(2n+1, F) : g^{\tau} = g^{-1}\bigg \}
\end{equation*}
and denote by $G_{n}$ the rank $n$ group belonging to the series which we fixed.

The set of standard parabolic subgroups will be fixed in a usual way, i.e., in $G_{n}$ we fix the minimal $F$-parabolic subgroup consisting of upper-triangular matrices in $G_{n}$. Then the Levi factors of standard parabolic subgroups have the form $M \cong GL(n_{1}, F) \times \cdots \times GL(n_{k}, F) \times G_{n'}$. For representations $\delta_{i}$ of $GL(n_{i}, F)$, $i = 1, 2, \ldots, k$, and a representation $\s$ of $G_{n'}$, the normalized parabolically induced representation Ind$_{M}^{G_{n}}(\delta_{1} \otimes \cdots \otimes \delta_{k} \otimes \s)$ will be denoted by $\delta_{1} \times \cdots \times \delta_{k} \rtimes \s$.

Let $R(G_{n})$ denote the Grothendieck group of the category of finite length representations of $G_{n}$ and define $R(G) = \oplus_{n \geq 0} R(G_{n})$. Similarly as in the case of a general linear group, sums of semisimplifications of Jacquet modules with respect to maximal parabolic subgroups define the mapping $\mu^{\ast} \colon R(G) \rightarrow R \otimes R(G)$.

Throughout the paper, the Jacquet module with respect to the smallest standard parabolic subgroup(s) admitting non-zero Jacquet modules for the representation in question will be called the minimal standard Jacquet module. For representation $\pi \in R(G_{n})$ with partial cuspidal support $\s \in R(G_{n'})$, the Jacquet module of $\pi$ with respect to the maximal parabolic subgroup having Levi factor equal to $GL(n-n', F) \times G_{n'}$ will be called the Jacquet module of $GL$-type and will be denoted by $s_{GL}(\pi)$. The sum of all irreducible constituents (counted with multiplicities) of $\mu^*(\pi)$ of the form $\tau\o\varphi$, where $\tau$ is cuspidal, will be denoted by $s_{top}(\pi)$.

We define $\kappa \colon R \otimes R \rightarrow R \otimes R$ by $\kappa (x \otimes y) = y \otimes x$ and extend contragredient $\widetilde{ }$ ~ to an automorphism of $R$ in the natural way. Let $M^{\ast} \colon R \rightarrow R$ be defined by \begin{equation*}
M^{\ast} = (m \otimes id) \circ (~ \widetilde{} \otimes m^{\ast}) \circ \kappa \circ m^{\ast}.
\end{equation*}

We recall the following formulas which hold for $\rho$ not necessary self-dual:
\begin{align*}
M^*(\delta([\nu^a\rho,\nu^b\rho])) 
&=
\sum_{i=a-1}^b \sum_{j=i}^b \delta([\nu^{-i}\tilde\rho,\nu^{-a}\tilde\rho]) \times
\delta([\nu^{j+1}\rho,\nu^b\rho]) \otimes \delta([\nu^{i+1}\rho,\nu^j\rho]) \\
\intertext{or}
M^*\left(\delta([\nu^{a}\rho,\nu^{b}\rho])\right) 
&= \sum_{k=0}^{b-a+1}
 \sum_{i=a-1}^{b-k}
\delta([\nu^{-i}\tilde\rho,\nu^{-a}\tilde\rho])
 \times
\delta([\nu^{k+i+1} \rho,\nu^{b}\rho]) \otimes
\delta([\nu^{i+1} \rho,\nu^{i+k}\rho]).
\end{align*}

The following lemma, which has been derived in \cite{T-Str}, presents a crucial structural formula for our calculations with Jacquet modules.

\begin{lemma} \label{osn}
Let $\rho$ be an irreducible cuspidal representation of $GL(m,F)$ and $a, b \in \mathbb{R}$ such that $b - a \in \mathbb{Z}_{\geq 0}$. For $\s \in R(G_{n})$ we write $\mu^{\ast}(\sigma) = \sum_{\tau, \sigma'} \tau \otimes \sigma'$. Then the following equalities hold:
\begin{align*}
\mu^\ast(\pi\rtimes\sigma) &= M^\ast(\pi)\rtimes\mu^\ast(\sigma)\\
\intertext{and}
\mu^\ast(\delta([\nu^a\rho,\nu^b\rho])\rtimes\sigma)
&=
\!\!\sum_{i=a-1}^b \sum_{j=i}^b \sum_{\tau,\sigma'}
\delta([\nu^{-i}\tilde\rho,\nu^{-a}\tilde\rho])\times\delta([\nu^{j+1}\rho,\nu^b\rho])\times\tau
\otimes{}\\
&\mkern140mu \otimes \delta([\nu^{i+1}\rho,\nu^j\rho])\rtimes\sigma'.
\end{align*}
We omit $\delta([\nu^{x} \rho, \nu^{y} \rho])$ if $x>y$.
\end{lemma}

We briefly recall the subrepresentation version of Langlands classification for general linear groups, which is necessary for determination of Jacquet modules of $GL$-type.

For every irreducible essentially square integrable representation $\delta$ of $GL(n, F)$, there exists an $e(\delta) \in \mathbb{R}$ such that $\nu^{-e(\delta)} \delta$ is unitarizable. Suppose that $\delta_{1}, \delta_{2}, \ldots, \delta_{k}$ are irreducible, essentially square integrable representations of $GL(n_{1}, F), GL(n_{2}, F), \ldots$, $GL(n_{k}, F)$ with $e(\delta_{1}) \leq e(\delta_{2}) \leq \ldots \leq e(\delta_{k})$. Then the induced representation $\delta_{1} \times \delta_{2} \times \cdots \times \delta_{k}$ has a unique irreducible subrepresentation, which we denote by $L(\delta_{1}, \delta_{2}, \ldots, \delta_{k})$. This irreducible subrepresentation is called the Langlands subrepresentation, and it appears with the multiplicity one in $\delta_{1} \times \delta_{2} \times \ldots \times \delta_{k}$. Every irreducible representation $\pi$ of $GL(n,F)$ is isomorphic to some $L(\delta_{1}, \delta_{2}, \ldots, \delta_{k})$. For a given $\pi$, the representations $\delta_{1}, \delta_{2}, \ldots, \delta_{k}$ are unique up to a permutation.

Also, throughout the paper we use the Langlands classification for classical groups and write a non-tempered irreducible representation $\pi$ of $G_{n}$ as the unique irreducible (Langlands) quotient of the induced representation of the form $\delta_{1} \times \delta_{2} \times \cdots \times \delta_{k} \rtimes \tau$, where $\tau$ is a tempered representation of $G_{t}$, and $\delta_{1}, \delta_{2}, \ldots, \delta_{k}$ are irreducible, essentially square integrable representations of $GL(n_{1}, F), GL(n_{2}, F), \ldots$, $GL(n_{k}, F)$ with $e(\delta_{1}) \geq e(\delta_{2}) \geq \ldots \geq e(\delta_{k}) > 0$. In this case, we write $\pi = L(\delta_{1}, \delta_{2}, \ldots, \delta_{k}; \tau)$. Again, for a given $\pi$, the representations $\delta_{1}, \delta_{2}, \ldots, \delta_{k}$ are unique up to a permutation.

Since the class of representations which we will study contains certain discrete series representations, we shortly recall basic ingredients of the classification of discrete series for classical groups due to M{\oe}glin and second author (\cite{Moe-Ex, Moe-T}).

According to this classification, discrete series are in bijective correspondence with admissible Jordan triples. More precisely, discrete series $\s$ of $G_{n}$ corresponds to the triple of the form $(\Jord, \s', \e)$, where $\s'$ is the partial cuspidal support of $\s$, $\Jord$ is the finite set (possibly empty) of pairs $(c, \rho)$, where $\rho$ is an irreducible cuspidal self-dual representation of $GL(n_{\rho}, F)$, and $c > 0$ an integer of appropriate parity, while $\epsilon$ is a function defined on a subset of $\Jord \cup (\Jord \times \Jord)$ and attains the values $1$ and $-1$.

For an irreducible cuspidal self-dual representation $\rho$ of $GL(n_{\rho}, F)$ we write $\Jord_{\rho} = \{ c : (c, \rho) \in \Jord \}$. If $\Jord_{\rho} \neq \emptyset$ and $c \in \Jord_{\rho}$, we put $c\_ = \max \{ d \in \Jord_{\rho}: d < c \}$, if it exists. Now, by definition $\e((c\_, \rho),(c,\rho)) = 1$ if there is some irreducible representation $\vf$ such that $\s$ is a subrepresentation of $\delta([\nu^{-\frac{c\_-1}{2}} \rho, \nu^{\frac{c-1}{2}} \rho]) \r \vf$.

In the classification mentioned above, every discrete series of $G_{n}$ is obtained inductively, starting from a Jordan triple of alternated  type which corresponds to a strongly positive representation, i.e., to the one whose all exponents in the supports of $GL$-type Jacquet module are positive. In each step one adds a pair of consecutive elements to Jordan block and the expanded $\e$-function equals one on that pair. For more details about this classification we refer the reader to \cite{T-inv} and \cite{T-temp}.

\section{Representations of segment type}

\begin{definition} Let $\rho$ and $\sigma$  be  irreducible cuspidal representations of a general linear group and of a classical group, respectively. Let $a,b\in\mathbb R$, $b-a\in\Z_{\geq 0}$, be such that
\begin{equation*}
0\leq a+b
\end{equation*}
and
\begin{equation*}
\d([\nu^{a } \rho,\nu^{b }\rho])\r\s
\end{equation*}
reduces. Then any irreducible subquotient of the above representation which contains $\d([\nu^{a } \rho,\nu^{b }\rho])\o\s$ in its Jacquet module (with respect to the standard parabolic subgroup), will be called a representation of segment type.
\end{definition}

We emphasize that representations $\d([\nu^{a } \rho,\nu^{b }\rho])\r\s$ and $\d([\nu^{-b} \rho,\nu^{-a}\rho])\r\s$ share the same composition series, but the choice $0 \leq a+b$ enables us to obtain, using the known formulas for Jacquet modules of $GL$-type, that the representation of segment type is always a subrepresentation of $\d([\nu^{a } \rho,\nu^{b }\rho])\r\s$.

Reducibility of the induced representation $\d([\nu^{a } \rho,\nu^{b }\rho])\r\s$ implies that $\rho$ is self-dual. Also, let $\a\in \mathbb R_{\geq0}$ be such that the induced representation $\nu^\a\rho\r\s$ reduces. Given $\rho$ and $\s$ such $\a $ is unique, by the results of Silberger \cite{Si}. Furthermore, recent results of Arthur imply that $\a\in(1/2)\Z$ (for more details we refer the reader to \cite{Art}).

Note that if $\d([\nu^{a } \rho,\nu^{b }\rho])\r\s$ is irreducible, then its Jacquet module contains $\d([\nu^{a } \rho,\nu^{b }\rho])\o\s$ and in this case we have a complete description of Jacquet modules with respect to the maximal parabolic subgroups of this representation.

Furthermore, it has been proved in \cite{T-irr} that $\d([\nu^{a } \rho,\nu^{b }\rho])\r\s$ reduces if and only if
\begin{equation*}
[\nu^{a } \rho,\nu^{b }\rho]\cap\{\nu^{-\a}\rho,\nu^{\a}\rho\}\ne \emptyset.
\end{equation*}
Thus, in the sequel we shall assume that $a,b\in (1/2)\Z$.

Also, we introduce the notion of \textit{proper} Langlands quotient of the induced representation $d$:
\begin{equation*}
L_{proper}(d)
=
\begin{cases}
L(d), & \text{if the corresponding standard module reduces};
\\
0, & \text {if the corresponding standard module is irreducible}.
\end{cases}
\end{equation*}

\begin{definition} \label{def1}
In the case of reducibility, we define $\delta([\nu^{a}\rho,\nu^{b}\rho]_+;\s)$ to be any irreducible subquotient of $\delta([\nu^{a}\rho,\nu^{b}\rho])\r\s$ which has in its minimal standard Jacquet module at least one irreducible subquotient whose all exponents are non-negative.
\end{definition}

In the sequel, we take $\d( [\nu^{a } \rho,\nu^{b }\rho]_-;\s)=0$
if $\d( [\nu^{a } \rho,\nu^{b }\rho])\r\s$ is a length two representation and if $ -a\ne b$.

In general case, the uniqueness of irreducible subquotients in Definition \ref{def1}, is provided by the following lemma.

\begin{lemma}
Suppose that $\a > 0$. There exists a unique irreducible subquotient of the induced representation $\d( [\nu^{a } \rho,\nu^{b }\rho])\r\s$ whose minimal standard Jacquet module contains at least one irreducible subquotient with all exponents being non-negative.
\end{lemma}
\begin{proof}
The claim obviously holds if the induced representation $\d( [\nu^{a } \rho,\nu^{b }\rho])\r\s$ is irreducible. Thus, we may assume $\{ -\a, \a \} \cap  [a, b] \neq \emptyset$. One can see directly that
\begin{equation*}
s_{GL}(\d( [\nu^{a} \rho,\nu^{b}\rho])\r\s) = \sum_{i=a}^{b+1} \d( [\nu^{-i+1} \rho,\nu^{-a}\rho]) \times \d( [\nu^{i} \rho,\nu^{b}\rho]) \otimes \s.
\end{equation*}
It follows immediately that if $a > 0$ then there is a unique irreducible subquotient of $s_{GL}(\d( [\nu^{a} \rho,\nu^{b}\rho])\r\s)$ with all exponents being non-negative, and we obtain such a subquotient for $i = a$. Similarly, if $a < 0$ and $a - \frac{1}{2} \in \mathbb{Z}$, we deduce that the unique irreducible subquotient of $s_{GL}(\d( [\nu^{a} \rho,\nu^{b}\rho])\r\s)$ having all exponents non-negative is obtained for $i = \frac{1}{2}$ (note that in this case the representation $\d( [\nu^{\frac{1}{2}} \rho,\nu^{-a}\rho]) \times \d( [\nu^{\frac{1}{2}} \rho,\nu^{b}\rho])$ is irreducible).

Thus, it remains to prove the lemma for $a \leq 0$, $a \in \mathbb{Z}$. Obviously, $\a \leq b$.

If $-a = b$, the only irreducible subquotients of $s_{GL}(\d( [\nu^{a} \rho,\nu^{b}\rho])\r\s)$ having all exponents non-negative is $\d( [\nu \rho,\nu^{b}\rho]) \times \d([\rho,\nu^{b}\rho]) \otimes \s$ (which appears with multiplicity two).

On the other hand, $\d( [\nu^{-b} \rho,\nu^{b}\rho])\r\s$ is a representation of the length two whose composition series consists of two non-isomorphic tempered representations. Using Lemma 4.1 of \cite{T-temp}, we deduce that there is a unique irreducible subquotient of $\d( [\nu^{-b} \rho,\nu^{b}\rho])\r\s$ having $\d( [\nu^{\a} \rho,\nu^{b}\rho]) \times \d([\nu^{\a} \rho,\nu^{b}\rho]) \otimes \d([\nu^{-\a+1} \rho,\nu^{\a-1}\rho]) \rtimes \s$ in its Jacquet module. Let us denote such subquotient by $\tau_{temp}$. Since $\d([\nu^{-\a+1} \rho,\nu^{\a-1}\rho]) \rtimes \s$ is irreducible, $\d( [\nu^{\a} \rho,\nu^{b}\rho]) \times \d([\nu^{\a} \rho,\nu^{b}\rho]) \otimes \d([\nu \rho,\nu^{\a-1}\rho]) \times \d([\rho,\nu^{\a-1}\rho]) \otimes \s$ appears in the Jacquet module of $\tau_{temp}$ with multiplicity two. Transitivity of Jacquet modules implies that there is some irreducible representation $\vf$ such that $\vf \otimes \s$ appears in the Jacquet module of $\tau_{temp}$ and $m^{\ast}(\vf) \geq \d( [\nu^{\a} \rho,\nu^{b}\rho]) \times \d([\nu^{\a} \rho,\nu^{b}\rho]) \otimes \d([\nu \rho,\nu^{\a-1}\rho]) \times \d([\rho,\nu^{\a-1}\rho])$. From cuspidal support of $\vf$ and structural formula for $\mu^{\ast}$ it follows easily that $\vf = \d( [\nu \rho,\nu^{b}\rho]) \times \d([\rho,\nu^{b}\rho])$. Since $\d( [\nu^{\a} \rho,\nu^{b}\rho]) \times \d([\nu^{\a} \rho,\nu^{b}\rho]) \otimes \d([\nu \rho,\nu^{\a-1}\rho]) \times \d([\rho,\nu^{\a-1}\rho]) $ appears in $m^{\ast}(\d( [\nu \rho,\nu^{b}\rho]) \times \d([\rho,\nu^{b}\rho]))$ with multiplicity one, it follows that the Jacquet module of $\tau_{temp}$ contains both copies of $\d( [\nu \rho,\nu^{b}\rho]) \times \d([\rho,\nu^{b}\rho]) \otimes \s$.

We now turn to the case $-a < b$.

In this case, irreducible subquotients of $s_{GL}(\d( [\nu^{a} \rho,\nu^{b}\rho])\r\s)$ having all exponents non-negative are $\d( [\nu \rho,\nu^{-a}\rho]) \times \d([\rho,\nu^{b}\rho]) \otimes \s$ (which appears with multiplicity two) and $L(\d( [\rho,\nu^{b}\rho]), \d( [\nu \rho,\nu^{b}\rho])) \otimes \s$ (which appears with multiplicity one).

Several possibilities, depending on $a$, will be considered separately.

Let us first assume $a \leq -\a$. By Theorem 2.1 of \cite{Mu-CSSP}, $\d( [\nu^{a} \rho,\nu^{b}\rho])\r\s$ is a length three representation and we denote by $\pi$ its discrete series subrepresentation whose corresponding $\e$-function $\e_{\pi}$ satisfies $\e_{\pi}((-2a+1,\rho),(2b+1,\rho)) = \e_{\pi}(((-2a+1)\_,\rho),(2a+1,\rho)) = 1$. If we denote by $\pi'$ a discrete series subrepresentation of $\d( [\nu^{a} \rho,\nu^{b}\rho])\r\s$ different than $\pi$, it easily follows that its $\e$-function $\e_{\pi'}$ satisfies $\e_{\pi'}((-2a+1,\rho),(2b+1,\rho)) = 1$ and $\e_{\pi}(((-2a+1)\_,\rho),(2a+1,\rho)) = -1$. Using Lemma 4.1 of \cite{Ma-JMDS1} and transitivity of Jacquet modules, we obtain that $\d( [\nu \rho,\nu^{-a}\rho]) \times \d([\rho,\nu^{b}\rho]) \otimes \s$ is not contained in the Jacquet module of the Langlands quotient of $\d( [\nu^{a} \rho,\nu^{b}\rho])\r\s$. Furthermore, using Proposition 7.2 of \cite{T-temp} and transitivity of Jacquet modules, we obtain that $\d( [\nu \rho,\nu^{-a}\rho]) \times \d([\rho,\nu^{b}\rho]) \otimes \s$ must be in the Jacquet module of $\pi$ and it is not contained in the Jacquet module of $\pi'$.

Condition $\e_{\pi}(((-2a+1)\_,\rho),(2a+1,\rho)) = 1$ implies that there is some irreducible representation $\varphi$ such that $\pi$ is a subrepresentation of $\d( [\rho,\nu^{-a}\rho]) \rtimes \varphi$. Using Frobenius reciprocity and formula for $\mu^{\ast}$, we deduce that $\vf$ is an irreducible subquotient of $\d( [\nu \rho,\nu^{b}\rho])\r\s$. Proposition 3.1\,(i) of \cite{Mu-CSSP} implies
\begin{equation*}
\d( [\nu \rho,\nu^{b}\rho])\r\s = L_{proper}(\d( [\nu \rho,\nu^{b}\rho]); \s) + L_{proper}(\d( [\nu \rho,\nu^{\a-1}\rho]); \s'),
\end{equation*}
where $\s'$ is the unique strongly positive discrete series subrepresentation of $\d( [\nu^{\a} \rho,\nu^{b}\rho]) \r \s$. Since $\pi$ is a square integrable representation, $\vf \neq L_{proper}(\d( [\nu \rho,\nu^{b}\rho]); \s)$. Consequently, $\vf$ equals $L_{proper}(\d( [\nu \rho,\nu^{\a-1}\rho]); \s')$ and it follows immediately that it is a subrepresentation of $\d( [\nu \rho,\nu^{b}\rho])\r\s$. Thus, Jacquet module of $\pi$ contains $\d( [\rho,\nu^{-a}\rho]) \otimes \d( [\nu \rho,\nu^{b}\rho])\otimes \s$. Since Jacquet module of $\d( [\nu \rho,\nu^{-a}\rho]) \times \d([\rho,\nu^{b}\rho])$ does not contain $\d( [\rho,\nu^{-a}\rho]) \otimes \d( [\nu \rho,\nu^{b}\rho])$, transitivity of Jacquet modules implies that $L(\d( [\rho,\nu^{-a}\rho]), \d( [\nu \rho,\nu^{b}\rho])) \otimes \s \leq \mu^{\ast}(\pi)$.

Let us now assume $-\a+1 < a$. In this case, by Theorem 4.1\,(ii) of \cite{Mu-CSSP}, $\d( [\nu^{a} \rho,\nu^{b}\rho])\r\s$ is a length two representation and we have
\begin{equation*}
\d( [\nu^{a} \rho,\nu^{b}\rho])\r\s = L_{proper}(\d( [\nu^{a} \rho,\nu^{b}\rho]); \s) + L_{proper}(\d( [\nu^{a} \rho,\nu^{\a-1}\rho]); \s'),
\end{equation*}
for the unique strongly positive discrete series subrepresentation $\s'$ of $\d( [\nu^{\a} \rho,\nu^{b}\rho])\r\s$. It can be directly verified that Jacquet module of $\d( [\nu^{a} \rho,\nu^{\a-1}\rho])\r\s'$ contains $\d( [\nu \rho,\nu^{-a}\rho]) \times \d([\rho,\nu^{b}\rho]) \otimes \s$ with multiplicity two and $L(\d( [\rho,\nu^{-a}\rho])$, $\d( [\nu \rho,\nu^{b}\rho])) \otimes \s$ with multiplicity one. Furthermore, applying Theorem 4.1\,(ii) of \cite{Mu-CSSP} to this induced representation, we deduce
\begin{equation*}
\d( [\nu^{a} \rho,\nu^{\a-1}\rho])\r\s' = L_{proper}(\d( [\nu^{a} \rho,\nu^{\a-1}\rho]); \s') + \tau,
\end{equation*}
where $\tau$ is the unique common irreducible subquotient of representations $\d( [\nu^{a} \rho,\nu^{\a-1}\rho])\r\s'$ and $\d( [\nu^{a} \rho,\nu^{\a-2}\rho])\r\s''$, for the unique strongly positive discrete series subrepresentation $\s''$ of $\nu^{\a-1} \times \d( [\nu^{\a} \rho,\nu^{b}\rho])\r\s$.

Note that Jacquet modules of both  $\d( [\nu \rho,\nu^{-a}\rho]) \times \d([\rho,\nu^{b}\rho])$ and $L(\d( [\rho,\nu^{-a}\rho])$, $\d( [\nu \rho,\nu^{b}\rho]))$ contain irreducible subquotients of the form $\nu^{-a} \rho \otimes \d( [\nu^{-a+1} \rho,\nu^{b}\rho]) \otimes \vf$.

We will show that Jacquet module of $\tau$ does not contain irreducible subquotients of the form $\d( [\nu \rho,\nu^{-a}\rho]) \times \d([\rho,\nu^{b}\rho]) \otimes \s$ or $L(\d( [\rho,\nu^{-a}\rho]), \d( [\nu \rho,\nu^{b}\rho])) \otimes \s$. Suppose, to the contrary, that one of these representations appears in $\mu^{\ast}(\tau)$. Then the Jacquet module of $\tau$ contains some irreducible subquotient of the form $\nu^{-a} \rho \otimes \tau'$, where $\mu^{\ast}(\tau')$ contains an irreducible constituent of the form $\d( [\nu^{-a+1} \rho,\nu^{b}\rho]) \otimes \vf'$. Since $\tau$ is an irreducible subquotient of $\d( [\nu^{a} \rho,\nu^{\a-2}\rho])\r\s''$, calculating $\mu^{\ast}(\d( [\nu^{a} \rho,\nu^{\a-2}\rho])\r\s'')$ we obtain that $\tau'$ is an irreducible subquotient of $\d( [\nu^{a+1} \rho,\nu^{\a-2}\rho])\r\s''$. Since $\mu^{\ast}(\tau') \geq \d( [\nu^{-a+1} \rho,\nu^{b}\rho]) \otimes \vf'$, using $a > -\a + 1$, we obtain that $\mu^{\ast}(\s'')$ contains an irreducible subquotient of the form $\d( [\nu^{\a - 1} \rho,\nu^{b}\rho]) \otimes \vf''$. Since $-a + 1 < \a$, this contradicts \cite{Ma-JMSP}, Theorem 5.3  (or Section 7 of this paper). Consequently, all irreducible constituents of $s_{GL}(\d( [\nu^{a} \rho,\nu^{b}\rho])\r\s)$ having all exponents non-negative are contained in Jacquet modules of $L_{proper}(\d( [\nu^{a} \rho,\nu^{\a-1}\rho]); \s')$.

It remains to consider the case $a = -\a + 1$. In this case, again by Theorem 4.1\,(ii) of \cite{Mu-CSSP}, we have
\begin{equation*}
\d( [\nu^{a} \rho,\nu^{b}\rho])\r\s = L_{proper}(\d( [\nu^{a} \rho,\nu^{b}\rho]); \s) + \tau_{temp},
\end{equation*}
where $\tau_{temp}$ is the unique common irreducible (tempered) subquotient of induced representations $\d( [\nu^{a} \rho,\nu^{b}\rho])\r\s$ and $\d( [\nu^{a} \rho,\nu^{-a}\rho])\r\s'$, where $\s'$ denotes the strongly positive discrete series subrepresentation of $\d( [\nu^{\a} \rho,\nu^{b}\rho])\r\s$.

Let us first assume $a \neq 0$. Obviously, there is an irreducible representation $\pi$ (resp., $\pi'$) such that $\d( [\rho,\nu^{-a}\rho]) \otimes \pi$ (resp., $\d( [\nu \rho,\nu^{-a}\rho]) \otimes \pi'$) is in the Jacquet module of $\tau_{temp}$. It is not hard to deduce that $\pi \leq \d( [\nu \rho,\nu^{b}\rho])\r\s$ (resp., $ \pi' \leq \d( [\rho,\nu^{b}\rho])\r\s$). From
\begin{align*}
\delta([\nu\rho,\nu^b\rho])\rtimes\sigma &=
L_{proper}(\delta([\nu\rho,\nu^b\rho]);\sigma) +
L_{proper}(\delta([\nu\rho,\nu^{-a}\rho]);\sigma') \\
\intertext{and}
\delta([\rho,\nu^b\rho])\rtimes\sigma &= L_{proper}(\delta([\rho,\nu^b\rho]);\sigma)
+ L_{proper}(\delta([\rho,\nu^{-a}\rho]);\sigma')
\end{align*}
for strongly positive discrete series subrepresentation $\s'$ of $\d( [\nu^{\a} \rho,\nu^{b}\rho])\r\s$, using temperedness of $\tau_{temp}$ in the same way as in the previous case, we get that $\mu^{\ast}(\pi) \geq \d( [\nu \rho,\nu^{b}\rho]) \otimes \s$ and $\mu^{\ast}(\pi') \geq \d( [\rho,\nu^{b}\rho])\otimes \s$ (with multiplicity two). Transitivity of Jacquet modules implies that all irreducible constituents of $s_{GL}(\d( [\nu^{a} \rho,\nu^{b}\rho])\r\s)$ having all exponents non-negative are contained in Jacquet modules of $\tau_{temp}$.

Now we assume $a = 0$ (i.e., $\a = 1$). In this case, since $\tau_{temp}$ is a subrepresentation of $\rho \times \d( [\nu \rho,\nu^{b}\rho])\r\s$, the representation $\rho \otimes \d( [\nu \rho,\nu^{b}\rho])\otimes \s$ is in the Jacquet module of $\tau_{temp}$ and it directly follows that $s_{GL}(\tau_{temp}) \geq L(\rho , \d( [\nu \rho,\nu^{b}\rho]))\otimes \s$. Furthermore, since $\tau_{temp}$ is a subrepresentation of $\d( [\rho,\nu^{b}\rho])\r\s$, using Lemma 4.7 and Corollary 4.9 from \cite{T-temp}, we deduce that $\d( [\nu \rho,\nu^{b}\rho])\otimes \rho \r\s$ is in the Jacquet module of $\tau_{temp}$. Since the representation $\rho \r\s$ is irreducible, an irreducible subquotient $\d( [\nu \rho,\nu^{b}\rho])\otimes \rho \otimes \s$ appears with multiplicity two in the Jacquet module of $\tau_{temp}$. Therefore, it easily follows that $\d( [\rho,\nu^{b}\rho])\otimes \s$ appears with multiplicity two in the Jacquet module of $\tau_{temp}$, and all irreducible constituents of $s_{GL}(\d( [\rho,\nu^{b}\rho])\r\s)$ having all exponents non-negative are contained in Jacquet modules of $\tau_{temp}$.

This proves the lemma.
\end{proof}

Representations of segment type in case $\{\nu^{-\a}\rho,\nu^{\a}\rho\}\subseteq [\nu^{a } \rho,\nu^{b }\rho]$, i.e.,
$[\nu^{-\a}\rho,\nu^{\a}\rho]\subseteq [\nu^{a } \rho,\nu^{b }\rho]$, have been carefully studied in \cite{T-seg}. We take a moment to recall the basic properties of these representations, which, in the case of strictly positive reducibility, we denote by $\d( [\nu^{a } \rho,\nu^{b }\rho]_\pm;\s)$. In this case they are always non-zero (in other cases which we shall consider below, $\d( [\nu^{a } \rho,\nu^{b }\rho]_\pm;\s)$ will be sometimes 0).

Then in the case of strictly positive reducibility  we have
\begin{enumerate}

\item[(S1)] If $\d( [\nu^{a } \rho,\nu^{b }\rho]_\pm;\s)\ne 0$, then its Jacquet module contains $\d( [\nu^{a } \rho,\nu^{b }\rho])\o\s$ (by definition).

\

\item[(S2)] $\d( [\nu^{a } \rho,\nu^{b }\rho]_+;\s)$ can be characterized as a subquotient of  $\d( [\nu^{a } \rho,\nu^{b }\rho])\r\s$ whose minimal standard Jacquet module contains at least one irreducible subquotient with all non-negative exponents.

\

\item[(S3)] If $\d( [\nu^{a } \rho,\nu^{b }\rho]_-;\s)\ne 0$, then it can be characterized as a subquotient of  $\d( [\nu^{a } \rho,\nu^{b }\rho])$ $\r\s$  which is a representation of segment type, and whose minimal standard Jacquet module contains no irreducible subquotients with all non-negative exponents.

\

\item[(S4)]  If $L_{proper}(\d( [\nu^{a } \rho,\nu^{b }\rho]);\s)\ne 0$, then it does not contain $\d( [\nu^{a } \rho,\nu^{b }\rho])\o\s$ in its  Jacquet module.

\

\item[(S5)]  Segment representations are subrepresentations of $\d( [\nu^{a } \rho,\nu^{b }\rho])\r\s$.

\end{enumerate}

We shall return to these representations later  (i.e., in case $[-\a,\a]\subseteq [a,b]$).

Now we introduce a convention which will be frequently used to keep our notation and statements of results uniform. \\
{\bf Convention}: Suppose $a,b\in(1/2)\Z$, $b-a, a+b\in\Z_{\geq0}$ are such that
\begin{equation*}
[\nu^{a}\rho,\nu^{b}\rho]\cap \{\nu^{-\a}\rho,\nu^{\a}\rho\}=\emptyset.
\end{equation*}
Then $\delta([\nu^{a}\rho,\nu^{b}\rho])\r\s$ is irreducible and we define
\begin{align*}
\delta([\nu^{a}\rho,\nu^{b}\rho]_+;\s) &=
\begin{cases}
0&\text{if } [\nu^{a}\rho,\nu^{b}\rho]\cap [\nu^{-\a}\rho,\nu^{\a}\rho]=\emptyset;
\\
 \delta([\nu^{a}\rho,\nu^{b}\rho])\r\s &\text{if } [\nu^{a}\rho,\nu^{b}\rho]\subseteq [\nu^{-\a+1}\rho,\nu^{\a-1}\rho].
\end{cases}
\intertext{Furthermore, we define}
L_\a(\d([\nu^{a}\rho,\nu^{b}\rho]);\s) &=
\begin{cases}
 \delta([\nu^{a}\rho,\nu^{b}\rho])\r\s &\text{if } [\nu^{a}\rho,\nu^{b}\rho]\cap [\nu^{-\a}\rho,\nu^{\a}\rho]=\emptyset;
\\
0 &\text{if } [\nu^{a}\rho,\nu^{b}\rho]\subseteq [\nu^{-\a+1}\rho,\nu^{\a-1}\rho].
\end{cases}
\end{align*}

In case
$$
[\nu^{a}\rho,\nu^{b}\rho]\cap \{\nu^{-\a}\rho,\nu^{\a}\rho\}\ne\emptyset,
$$
we have already defined $\delta([\nu^{a}\rho,\nu^{b}\rho]_+;\s)$, and we simply set
$$
L_\a(\d([\nu^{a}\rho,\nu^{b}\rho]);\s)=L(\d([\nu^{a}\rho,\nu^{b}\rho]);\s).
$$

In the sequel, positive $\a\in (1/2)\Z$ will always denote the reducibility exponent for $\rho$ and $\sigma$.

We now proceed with the study of Jacquet modules of the representations of segment type, with respect to the maximal parabolic subgroups. Several possible cases will be studied separately. We emphasize that the convention which we have introduced enables us to state the results, obtained in different cases, in a uniform way.

\section{Representations of segment type
corresponding  to segments  not containing $[\nu^{-\a } \rho,\nu^{\a }\rho]$}

In this section we start our determination of Jacquet modules for representations of segment type. First we study such representations attached to segments which contain exactly one of $\{ \a, -\a \}$.

In what follows, we shall denote segments by
$
 [\nu^{-c } \rho,\nu^{d }\rho].
$
Clearly, $c$ (and $d$) must satisfy $c-\a\in\Z$. Furthermore, we assume
$
|c|\leq d.
$

In this section we consider the case when
\begin{equation*}
[\nu^{-\a } \rho,\nu^{\a }\rho]\not\subseteq [\nu^{-c } \rho,\nu^{d }\rho].
\end{equation*}

\begin{theorem} \label{tmprvi} Let $c,d\in(1/2)\Z$ be such that $d+c,d-c\in\Z_{\geq 0}$, $d-\a\in\Z$ and,
$$
-\a<-c\leq \a\leq d.
$$
Then
\begin{enumerate}

\item[$(A1)$] $\delta([\nu^{-c}\rho,\nu^{d}\rho])\r\s$ is a representation of length two, and the composition series consists of $\delta([\nu^{-c}\rho,\nu^{d}\rho]_+;\s)$ and $L(\delta([\nu^{-c}\rho,\nu^{d}\rho]);\s)$.

\item[$(A2)$] For $-c=\a$, $\delta([\nu^{-c}\rho,\nu^{d}\rho]_+;\s)$ is square integrable. For $-c=-\a+1$, the representation is tempered, but not square integrable. For $-\a+1<-c$, we have
\begin{equation*}
\delta([\nu^{-c}\rho,\nu^{d}\rho]_+;\s)=L(\delta([\nu^{-c}\rho,\nu^{\a-1}\rho]);\delta([\nu^{\a}\rho,\nu^{d}\rho]_+;\s)).
\end{equation*}

\item[$(A3)$] If $-c<d$, then

 $s_{top} (\delta([\nu^{-c}\rho,\nu^{d}\rho]_+;\s))= \nu^{d}\rho\o\delta([\nu^{-c}\rho,\nu^{d-1}\rho]_+;\s) + \nu^{c}\rho\o\delta([\nu^{-c+1}\rho,\nu^{d}\rho]_+;\s).
$\\
For $-c=d \ (=\a)$, we have $s_{top} (\d(\nu^{\a}\rho_+;\s))= \nu^{\a}\rho\o\s$.

\item[$(A4)$] If $-c<d$, then
\begin{align*}
s_{top}(L(\delta([\nu^{-c}\rho,\nu^{d}\rho]);\s)) & =  \nu^{d}\rho\o L_{\a}(\delta([\nu^{-c}\rho,\nu^{d-1}\rho]);\s) + {} \\
& \qquad + \nu^{c}\rho\o L_{\a}(\delta([\nu^{-c+1}\rho,\nu^{d}\rho]);\s).
\end{align*}\\
For $-c=d \ (=\a)$, we have $s_{top} (L(\nu^{\a}\rho;\s))= \nu^{-\a}\rho\o\s$.

\item[$(A5)$] $
s_{GL}(\delta([\nu^{-c}\rho,\nu^{d}\rho]_+;\s))=\sum_{i=-c-1}^{\a-1}
\delta([\nu^{-i}\rho,\nu^{c}\rho]) \t \delta([\nu^{i+1}\rho,\nu^{d}\rho])\o\s.
$

\item[$(A6)$] For $-c<d$ we have
\begin{align*}
s_{GL}(L(\delta([\nu^{-c}\rho,\nu^{d}\rho]);\s)) & = \sum_{i=\a}^{d}
\delta([\nu^{-i}\rho,\nu^{c}\rho]) \t \delta([\nu^{i+1}\rho,\nu^{d}\rho])\o\s \\
& = \sum_{i=\a}^{d} L(\delta([\nu^{-i}\rho,\nu^{c}\rho]) , \delta([\nu^{i+1}\rho,\nu^{d}\rho]))\o\s.
\end{align*}

\end{enumerate}
\end{theorem}

\begin{proof} For $-c=d \ (=\a)$, the theorem holds true (this is a very well known and simple fact). Therefore, in the proof we consider only the case
$
-c<d.
$

Recall that
\begin{equation} \label{s1}
s_{GL}(\delta([\nu^{-c}\rho,\nu^{d}\rho])\r\s) = \sum_{i=-c-1}^{d}
\delta([\nu^{-i}\rho,\nu^{c}\rho]) \t \delta([\nu^{i+1}\rho,\nu^{d}\rho])\o\s.
\end{equation}
At this point, we shall show that $(A2)$ follows from $(A5)$ and results obtained in \cite{Mu-CSSP}. The square integrability in case $-c = \a$ follows from the known fact that in this case $\delta([\nu^{-c}\rho,\nu^{d}\rho]_+;\s)$ is strongly positive, while temperedness in case $-c = - \a +1$ follows directly from $(A5)$. In the remaining cases it can be seen directly from Proposition 3.1\,(i) of \cite{Mu-CSSP} and Theorem 4.1\,(ii) of the same paper that we have a length two representation and the Langlands data as given in $(A2)$ (note that the representation $\d([\nu^\a\rho,\nu^{d}\rho]; \sigma)$ is strongly positive).

Observe that for any $\a+1\leq j \leq d+1$ we have
\begin{align*}
L(\d([\nu^{-c}\rho,\nu^{d}\rho]);\s) & \h \d([\nu^{-d}\rho,\nu^{c}\rho])\r\s \\
& \h \d([\nu^{-j+1}\rho,\nu^{c}\rho]) \t \d([\nu^{-d}\rho,\nu^{-j}\rho])\r\s \\
& \cong \d([\nu^{-j+1}\rho,\nu^{c}\rho]) \t \d([\nu^{j}\rho,\nu^{d}\rho])\r\s,
\end{align*}
since the representation $\d([\nu^{j}\rho,\nu^{d}\rho])\r\s$ is irreducible (\cite{T-irr}). Now in $s_{GL}(L(\delta([\nu^{c}\rho,\nu^{d}\rho]);\s))$ we must have representations with cuspidal supports which follow from all embeddings of the above type. Such representations are among subquotients of \eqref{s1} and considered cuspidal supports show up precisely in the following part of \eqref{s1}:
\begin{equation*}
\sum_{i=\a}^{d}
\delta([\nu^{-i}\rho,\nu^{c}\rho]) \t \delta([\nu^{i+1}\rho,\nu^{d}\rho])\o\s.
\end{equation*}
Observe that all representations in the above sum are irreducible. This implies
\begin{equation} \label{form1}
s_{GL}(L(\delta([\nu^{c}\rho,\nu^{d}\rho]);\s))\geq \sum_{i=\a}^{d}
\delta([\nu^{-i}\rho,\nu^{c}\rho]) \t \delta([\nu^{i+1}\rho,\nu^{d}\rho])\o\s.
\end{equation}
The formula for $\mu^*$ gives
\begin{equation} \label{form2}
s_{top} (\delta([\nu^{-c}\rho,\nu^{d}\rho])\r\s))= \nu^{d}\rho\o\delta([\nu^{-c}\rho,\nu^{d-1}\rho])\r\s + \nu^{c}\rho\o\delta([\nu^{-c+1}\rho,\nu^{d}\rho])\r\s.
\end{equation}

Now we turn to the proof of Theorem \ref{tmprvi}.

Note that in case $-c=\a$, $(A1)$ again follows from Proposition 3.1\,(i), while $(A2)$, $(A3)$ and $(A5)$ are direct consequences of Theorem 4.6 of \cite{Ma-SP} and Theorem 5.3 of \cite{Ma-JMSP}. One gets $(A4)$ and $(A6)$ using, additionally, the formulas (\ref{form1}) and (\ref{form2}). This completes the proof for this case. Therefore, in what follows, we shall always assume
\begin{equation*}
-\a+1\leq -c\leq \a-1.
\end{equation*}

Now we shall prove the theorem for $d=\a$. The proof will be by induction on $-c$. For $-c=\a$ we have observed that the theorem holds true. We fix $-c$ as above and assume that the theorem holds for $-c+1$ (and $d=\a$; clearly, $c<\a$). The inductive assumption implies
\begin{multline*}
s_{top} (\delta([\nu^{-c}\rho,\nu^{\a}\rho])\r\s))= \nu^{\a}\rho\o\delta([\nu^{-c}\rho,\nu^{\a-1}\rho]_+;\s)
+ \\
+ \nu^{c}\rho\o\delta([\nu^{-c+1}\rho,\nu^{\a}\rho]_+;\s)
 + \nu^{c}\rho\o L_\a(\delta([\nu^{-c+1}\rho,\nu^{\a}\rho]);\s).
\end{multline*}
Obviously, the first summand must be in the Jacquet  module of $\delta([\nu^{-c+1}\rho,\nu^{\a}\rho]_+;\s)$, while the last summand must be in the Jacquet module of the Langlands quotient. Observe that, if $c$ is non-negative, the second summand must be in the Jacquet module of $\delta([\nu^{-c}\rho,\nu^{\a}\rho]_+;\s)$. We prove that this also holds in the case $c<0$. Obviously, the term $\delta([\nu^{-c}\rho,\nu^{\a}\rho])\o\s$ in \eqref{s1} is the only term with such cuspidal support. Thus
\begin{align*}
\delta([\nu^{-c}\rho,\nu^\alpha\rho]_+;\sigma)
&\hookrightarrow \delta([\nu^{-c}\rho,\nu^\alpha\rho]) \rtimes \sigma \\
&\hookrightarrow  \delta([\nu^{-c+1}\rho,\nu^\alpha\rho])\times\nu^{-c}\rtimes\sigma \\
&\cong \delta([\nu^{-c+1}\rho,\nu^\alpha\rho])\times\nu^c\rho\rtimes\sigma \\
&\cong \nu^c\rho\times\delta([\nu^{-c+1}\rho,\nu^\alpha\rho])\rtimes\s.
\end{align*}
Using Frobenius reciprocity and transitivity of Jacquet modules we deduce that the second summand  in $s_{top} (\delta([\nu^{-c}\rho,\nu^{\a}\rho])\r\s)
$ is in the Jacquet module of $\delta([\nu^{-c}\rho,\nu^{\a}\rho]_+;\s)$.

This analysis of $s_{top} (\delta([\nu^{-c}\rho,\nu^{\a}\rho])\r\s)
$ implies length two of $\delta([\nu^{-c}\rho,\nu^{\a}\rho])\r\s
$, and the rest of $(A1)$. Also, it implies $(A3)$ and $(A4)$. Observe that we have just proved
\begin{equation*}
s_{top}(L_\a(\delta([\nu^{-c}\rho,\nu^{\a}\rho]);\s))=\nu^{c}\rho\o L_\a(\delta([\nu^{-c+1}\rho,\nu^{\a}\rho]);\s),
\end{equation*}
while the inductive assumption implies
\begin{equation*}
s_{GL}(L_\a(\delta([\nu^{-c+1}\rho,\nu^{\a}\rho]);\s))= \delta([\nu^{-\a}\rho,\nu^{c-1}\rho])\o\s.
\end{equation*}
This implies that the minimal Jacquet module of $L_\a(\delta([\nu^{-c+1}\rho,\nu^{\a}\rho]);\s)$ is $\nu^{c}\rho\o\nu^{c-1}\rho\o\cdots\o \nu^{-\a}\rho\o\s$. This yields $(A6)$, which directly implies $(A5)$ using $(A1)$ and \eqref{s1}. This completes the case $d=\a$.

It remains to consider the case
\begin{equation*}
-\a<-c<\a<d.
\end{equation*}
We shall prove the theorem by induction on $d$. Observe that the theorem holds for $d=\a$. Fix $d>\a$ and suppose that the theorem holds for $d-1$. To prove the theorem for $d$, we proceed by induction on $-c$. Recall that the theorem holds for $-c=-\a$. Fix $-c>-\a$, and suppose that the theorem holds for $-c+1$.

We consider two cases. First we discuss the case
\begin{equation*}
0<-c.
\end{equation*}
 Observe that in the sum
\begin{equation*}
s_{GL}(\delta([\nu^{-c}\rho,\nu^{d}\rho])\r\s)
=
\sum_{i=-c-1}^{d}
\delta([\nu^{-i}\rho,\nu^{c}\rho]) \t \delta([\nu^{i+1}\rho,\nu^{d}\rho])\o\s
\end{equation*}
all the summands are irreducible since we have $c < i-1$. Also, all the summands have different cuspidal supports. In this case $\delta([\nu^{-c}\rho,\nu^{d}\rho]_+,
\sigma)$ is the subquotient of $\delta([\nu^{-c}\rho,\nu^{d}\rho])\r\s$  which has $\delta([\nu^{-c}\rho,\nu^{d}\rho])\o \sigma$ in its Jacquet module. Observe that the last representation has multiplicity one in the full Jacquet module, and it is a direct summand in the Jacquet module (other terms have different cuspidal support). This implies
\begin{align*}
\delta([\nu^{-c}\rho,\nu^d\rho]_+,\sigma)
&\hookrightarrow
\nu^d\rho\times\cdots\times\nu^{-c+1}\rho\times\nu^{-c}\rho\rtimes\sigma \\
&\cong \nu^d\rho\times\cdots\times\nu^{-c+1}\rho\times\nu^c\rho\rtimes\sigma \\
&\cong \dots
\cong  \nu^c\rho\times\nu^d\rho\times\cdots\times\nu^{-c+1}\rho\rtimes\sigma \\
&\cong \dots
\cong
\nu^c\rho\times\nu^{c-1}\rho\times\cdots\times\nu^{\alpha+1}\rho\times\nu^d\rho\times\cdots\times\nu^\alpha\rho\rtimes\sigma.
\end{align*}

Now above observations regarding Jacquet modules and cuspidal supports imply that $\delta([\nu^{-c}\rho,\nu^{d}\rho]_+,
\sigma)$ contains in its Jacquet module of $GL$-type representations with the same $GL$-cuspidal support as the following representations:
\begin{equation*}
\delta([\nu^{-i+1}\rho,\nu^{c}\rho]) \t \delta([\nu^{i}\rho,\nu^{d}\rho]), \       -c\leq i \leq  \a.
\end{equation*}
This implies that
\begin{equation*}
\sum_{i=-c-1}^{\a-1}\delta([\nu^{-i}\rho,\nu^{c}\rho]) \t \delta([\nu^{i+1}\rho,\nu^{d}\rho]) \otimes \s \leq s_{GL}(\delta([\nu^{-c}\rho,\nu^{d}\rho]_+,
\sigma)).
\end{equation*}

We have already seen that
\begin{equation} \label{in1}
\sum_{i=\a}^{d}
 \delta([\nu^{-i}\rho,\nu^{c}\rho]) \t \delta([\nu^{i+1}\rho,\nu^{d}\rho]) \o\s
 \leq
 s_{GL}(L(\d([\nu^{-c}\rho,\nu^{d}\rho]);\s)).
\end{equation}
Since the sum of the left-hand sides of previous two formulas is the whole Jacquet module of $GL$-type of $\delta([\nu^{-c}\rho,\nu^{d}\rho])\r\s$, we conclude that $\delta([\nu^{-c}\rho,\nu^{d}\rho])\r\s$ is a representation of length two and in the above two inequalities we actually have equalities.

Hence, we have proved $(A1)$, $(A5)$ and $(A6)$.

Note that the inductive assumptions imply
\begin{multline*}
s_{top} (\delta([\nu^{-c}\rho,\nu^{d}\rho])\r\s))= \nu^{d}\rho\o\delta([\nu^{-c}\rho,\nu^{d-1}\rho]_+;\s) + \nu^{d}\rho\o L_\a(\delta([\nu^{-c}\rho,\nu^{d-1}\rho]);\s) \\
+ \nu^{c}\rho\o\delta([\nu^{-c+1}\rho,\nu^{d}\rho]_+;\s)
 + \nu^{c}\rho\o L_\a(\delta([\nu^{-c+1}\rho,\nu^{d}\rho]);\s).
\end{multline*}
Obviously, the first summand belongs to Jacquet module of $\delta([\nu^{-c}\rho,\nu^{d}\rho]_+;\s)$, while it can be seen, using the same argument as before, that the last one belongs to Jacquet module of the Langlands quotient. From $(A5)$ it follows that $\delta([\nu^{-c}\rho,\nu^{d}\rho]_+;\s)$ must have the third summand in its Jacquet module (in Jacquet module we have to have a subquotient of the form $\nu^c\rho\o\dots\o\s$). Similarly, $(A6)$ implies that the second term must be in the Jacquet module of the Langlands quotient. This proves $(A3)$ and $(A4)$.

It remains to prove the theorem in the case when
\begin{equation*}
-\a<-c\leq 0 <\a<d.
\end{equation*}
From the inductive assumption we know that
\begin{multline*}
s_{top} (\delta([\nu^{-c}\rho,\nu^{d}\rho])\r\s))= \nu^{d}\rho\o\delta([\nu^{-c}\rho,\nu^{d-1}\rho]_+;\s) + \nu^{d}\rho\o L_\a(\delta([\nu^{-c}\rho,\nu^{d-1}\rho]);\s) \\
+
\nu^{c}\rho\o\delta([\nu^{-c+1}\rho,\nu^{d}\rho]_+;\s)
 + \nu^{c}\rho\o L_\a(\delta([\nu^{-c+1}\rho,\nu^{d}\rho]);\s).
\end{multline*}
Observe that all four summands above are non-zero. Looking at the most positive terms, we see that the first and the third term must come from $\delta([\nu^{-c}\rho,\nu^{d}\rho]_+;\s)$, while the last one comes from the Langlands quotient. It remains to determine where does the second summand belong.

Recall that the Langlands quotient embeds into
\begin{align*}
L(\delta([\nu^{c}\rho,\nu^{-d}\rho]);\s) & \h \nu^{c}\rho\t \cdots \t \nu^{-d}\rho\r\s \\
& \cong \nu^{c}\rho\t \cdots\t \nu^{d}\r\s  \\
& \cong \nu^{d}\rho\t \nu^{c}\rho\t \cdots\t \nu^{-d+1}\rho\r\s,
\end{align*}
since $\a< d$, $-\a+1\leq-c$, and $\a<d$. This shows that the second summand is in Jacquet module of the Langlands quotient. In this way we get the proof of $(A1)$, $(A3)$ and $(A4)$.

Observe that $(A5)$ holds if and only if $(A6)$ holds (this is a consequence of \eqref{s1} and the length two of $\delta([\nu^{c}\rho,\nu^{d}\rho])\r\s$). Furthermore, by the above estimate \eqref{in1}, we also know that the left-hand side of $(A6)$ contains the right-hand side.

Let $\rho$ be a representation of $GL(p,F)$. We show that an equality holds on the level of semi simplifications of Jacquet modules $r_{(p,(d+c)p)}$ (from Bernstein-Zelevinsky $GL$-classification) when applied to the left- and right-hand sides of $(A6)$. The result for the left-hand side follows from the inductive assumption $(A4)$.
Using this assumption, we deduce that the corresponding Jacquet module of the left-hand side equals
\begin{multline*}
\sum_{i=\a}^{d-1}
\nu^d\rho\o\delta([\nu^{-i}\rho,\nu^{c}\rho]) \t \delta([\nu^{i+1}\rho,\nu^{d-1}\rho])\o\s \\
+\sum_{i=\a}^{d} \nu^c\rho\o
\delta([\nu^{-i}\rho,\nu^{c-1}\rho]) \t \delta([\nu^{i+1}\rho,\nu^{d}\rho])\o\s.
\end{multline*}
Applying the Jacquet module $r_{(p,(d+c)p)}$ to the right-hand side of $(A6)$ we obtain
\begin{multline*}
\bigg(\sum_{i=\a}^{d} r_{(p,(d+c)p)}(
\delta([\nu^{-i}\rho,\nu^{c}\rho]) \t \delta([\nu^{i+1}\rho,\nu^{d}\rho]))\bigg)\o\s \\
=\bigg(\sum_{i=\a}^{d} \nu^c\rho\o
\delta([\nu^{-i}\rho,\nu^{c-1}\rho]) \t \delta([\nu^{i+1}\rho,\nu^{d}\rho])\bigg)\o\s \\
+\bigg(\sum_{i=\a}^{d-1} \nu^d\rho\o
\delta([\nu^{-i}\rho,\nu^{c}\rho]) \t \delta([\nu^{i+1}\rho,\nu^{d-1}\rho])\bigg)\o\s.
\end{multline*}
Observe that we have proved the equality. Therefore, $(A6)$, and $(A5)$, hold true.

Note that in this case $\delta([\nu^{-c}\rho,\nu^{d}\rho])\o\s$ has multiplicity two in the  Jacquet module of $(A5)$.

The proof is now complete.
\end{proof}

In the case when a discrete series subquotient appears, we have the following:

\begin{remark}
If $-\a<-c\leq \a\leq d$ and $\delta([\nu^{-c}\rho,\nu^{d}\rho]_+;\s)$ is square integrable, i.e., $-c = \a$, we note that $\delta([\nu^{-c}\rho,\nu^{d}\rho]_+;\s)$ is strongly positive. Furthermore, $\Jord(\delta([\nu^{-c}\rho,\nu^{d}\rho]_+;\s)) =$ $\Jord(\s) \setminus \{ (2 \a - 1, \rho) \} \cup \{ (2d+1, \rho)\}$.
\end{remark}

In the sequel, we shall symmetrize notation, i.e., we will define $\delta([\nu^{-d}\rho,\nu^{c}\rho]_+;\s)$ to be $\delta([\nu^{-c}\rho,\nu^{d}\rho]_+;\s)$, $L(\delta([\nu^{-d}\rho,\nu^{c}\rho]);\s)$ to be $L(\delta([\nu^{-c}\rho,\nu^{d}\rho]);\s)$, etc.

Assume $-\a<-c\leq \a\leq d$. Consider
\begin{multline*}
\mu^*\left(\delta([\nu^{-c}\rho,\nu^{d}\rho])\r\s\right) =
\sum_{i= -c-1}^{ d}
\sum_{j=i}^{ d}
\delta([\nu^{-i}\rho,\nu^{c}\rho])
 \times
\delta([\nu^{j+1} \rho,\nu^{d}\rho]) \otimes
\delta([\nu^{i+1} \rho,\nu^{j}\rho])\r\s \\
=
\sum_{i= -c-1}^{ \a-1}
\sum_{j=i}^{ d}
\delta([\nu^{-i}\rho,\nu^{c}\rho])
 \times
\delta([\nu^{j+1} \rho,\nu^{d}\rho]) \otimes
\delta([\nu^{i+1} \rho,\nu^{j}\rho])\r\s + \\
 +
\sum_{i= \a}^{ d}
\sum_{j=i}^{ d}
\delta([\nu^{-i}\rho,\nu^{c}\rho])
 \times
\delta([\nu^{j+1} \rho,\nu^{d}\rho]) \otimes
\delta([\nu^{i+1} \rho,\nu^{j}\rho])\r\s.
\end{multline*}
Observe that if some exponent appearing in the cuspidal support on the left-hand side of the tensor product, is less than or equal to $- \a$, then $[-\a,\a]\cap [i+1,j]=\emptyset.$

\begin{corollary} \label{cor1}
Let $c,d\in(1/2)\Z$ such that $d+c,d-c\in\Z_{\geq 0}$, $d-\a\in\Z$ and $-\a<-c\leq \a\leq d$. We have
\begin{align*}
\mu^*\big(&\delta([\nu^{-c}\rho,\nu^{d}\rho]_+;\s)\big) 
\\
& = \sum_{i= -c -1}^{ d-1%
}\sum_{j=i+1}^{ d} \delta([\nu^{-i}\rho,\nu^{c}\rho])  \times \delta([\nu^{j+1}\rho,\nu^{d}\rho]) \otimes
\delta([\nu^{i+1} \rho,\nu^{j}\rho]_+;\s) +{} 
\\
& + \mkern-20mu\sum_{-c-1\le i\le c-1}\ \sum_{i+1\le j\le
c}\mkern-75mu\rule[-4.5ex]{0pt}{2ex}_{i+j < -1}\mkern25mu
\delta([\nu^{-i}\rho,\nu^{c}\rho])
 \times
\delta([\nu^{j+1} \rho,\nu^{d}\rho]) \otimes
L_\a(\delta([\nu^{i+1} \rho,\nu^{j}\rho]);\s) +{} \\[-1.5ex]
&\mkern100mu+ \mkern-10mu\sum_{i=-c-1}^{\a - 1}
\delta([\nu^{-i}\rho,\nu^{c}\rho]) \t \delta([\nu^{i+1}\rho,\nu^{d}\rho])\o\s,
\end{align*}
and
\begin{align*}
\mu^*\big(&L(\delta([\nu^{-c}\rho,\nu^d\rho]);\sigma)\big) 
\\
&= \mu^\ast\big(L_\alpha(\delta([\nu^{-c}\rho,\nu^d\rho]);\sigma)\big) = \mu^\ast\big(L_{proper}(\delta([\nu^{-c}\rho,\nu^d\rho]);\sigma)\big)  
\\
&= \mkern-20mu\sum_{-c-1\le i\le d-1%
}\ \sum_{i+1\le j\le
d}\mkern-75mu\rule[-4.5ex]{0pt}{2ex}_{0\le i+j}\mkern25mu
L(\delta([\nu^{-i}\rho,\nu^c\rho]),\delta([\nu^{j+1}\rho,\nu^d\rho])\big) \otimes
L_\alpha(\delta([\nu^{i+1}\rho,\nu^j\rho]);\sigma)) +{}  
\\[-1.5ex]
&\mkern100mu+ \mkern-10mu\sum_{i=\alpha}^d
L(\delta([\nu^{-i}\rho,\nu^c\rho]),\delta([\nu^{i+1}\rho, \nu^d\rho])) \otimes
\sigma. 
\end{align*}
\end{corollary}

\begin{proof} Note that, by our convention introduced earlier, the second sum appearing in $\mu^*\left(\delta([\nu^{-c}\rho,\nu^{d}\rho]_+;\s)\right)$ is zero, while the first sum appearing in $\mu^*\left(L(\delta([\nu^{-c}\rho,\nu^{d}\rho]);\s)\right)$ equals
$$  \sum_{i=-c-1}^{d} \sum_{j=i+1}^{d} \delta([\nu^{-i}\rho,\nu^{c}\rho]) \times
\delta([\nu^{j+1} \rho,\nu^{d}\rho]) \otimes
L_\a(\delta([\nu^{i+1} \rho,\nu^{j}\rho]);\s)).$$
However, using our convention, the statements of our main results regarding Jacquet modules look essentially identical in all the cases considered.

The following identity is a direct consequence of the structural formula:
\begin{multline*}   
\mu^*(\delta([\nu^{-c}\rho,\nu^d\rho])\rtimes\delta) 
=\sum_{i=-c-1}^d \sum_{j=i}^d
\delta([\nu^{-i}\rho,\nu^c\rho])\times\delta([\nu^{j+1}\rho,\nu^d\rho]) \otimes{}\\
\otimes
\big( \delta\big([\nu^{i+1}\rho,\nu^j\rho]_+;\sigma)+L_\alpha(\delta([\nu^{i+1}\rho,\nu^j\rho]);\sigma)\big)
.
\end{multline*}

We shall first analyze the case when $\delta([\nu^{i+1} \rho,\nu^{j}\rho])\r\s$ is irreducible. Since $-\a + 1 \leq -c\leq i\leq d$, irreducibility can happen only in the following two cases.

The first case is $j<\a$. Then we have $\delta([\nu^{i+1} \rho,\nu^{j}\rho]_+;\s)=\delta([\nu^{i+1} \rho,\nu^{j}\rho])\r\s$ and $L_\a(\delta([\nu^{i+1} \rho,\nu^{j}\rho]);\s)=0$ (by our convention). Now it is enough to prove that the Jacquet module of Langlands quotient cannot contain
\begin{equation*}
\delta([\nu^{-i}\rho,\nu^{c}\rho])
 \times
\delta([\nu^{j+1} \rho,\nu^{d}\rho]) \otimes\delta([\nu^{i+1} \rho,\nu^{j}\rho])\r\s.
\end{equation*}
Suppose it does. Then the formula for $s_{GL}(\delta([\nu^{i+1} \rho,\nu^{j}\rho])\r\s)$ implies that in the  minimal non-zero standard Jacquet module we would have at least one term of the form $\dots \o \nu^{x}\rho\o\s$ for $- \a + 1 \leq x \leq \a$. On the other hand, $(A6)$ tells us that such subquotients for Jacquet module of the Langlands quotient can be only of the form
\begin{equation*}
\dots \o \nu^{-i+1}\rho\o\s \text{ \ where \ } \ \a+1\leq i \leq d+1,   \text{ \ or \ } \dots \o \nu^{i}\rho\o\s,  \text{ \ where \ } \ \a+1\leq i \leq d.
\end{equation*}
Obviously, the term of the form $\dots \o \nu^{x}\rho\o\s$, for $- \a + 1 \leq x \leq \a$, cannot be among them, and we get a contradiction.

The second case when one has irreducibility of $\delta([\nu^{i+1} \rho,\nu^{j}\rho])\r\s$ is when
\begin{equation*}
\a<i+1.
\end{equation*}
Then $\delta([\nu^{i+1} \rho,\nu^{j}\rho]_+;\s)=0$ and $L_\a(\delta([\nu^{i+1} \rho,\nu^{j}\rho]);\s)=\delta([\nu^{i+1} \rho,\nu^{j}\rho])\r\s$ (by our convention). Now it is enough to prove that
\begin{equation*}
\delta([\nu^{-i}\rho,\nu^{c}\rho])
 \times
\delta([\nu^{j+1} \rho,\nu^{d}\rho]) \otimes\delta([\nu^{i+1} \rho,\nu^{j}\rho])\r\s
\end{equation*}
cannot occur in the Jacquet module of $\delta([\nu^{-c} \rho,\nu^{d}\rho]_+;\s)$. Suppose it does.
Then in the minimal non-zero standard Jacquet module of $\delta([\nu^{-c} \rho,\nu^{d}\rho]_+;\s)$ we would have a term of the form $\dots\o\nu^{i+1}\rho\o\s$, where $\a<i+1$. On the other hand, the formula in $(A5)$ implies that $\delta([\nu^{-c} \rho,\nu^{d}\rho]_+;\s)$ can have in that Jacquet module only the terms of the form
\begin{equation*}
\dots\o \nu^{-i'+1}\o\s \text{ \ where \ } -c+1\leq i'\leq \a \text{\quad or \quad} \dots\o \nu^{i'}\o\s \text{ \ where \ }-c\leq i'\leq \a .
\end{equation*}
Obviously, the above term of the form $\dots\o\nu^{i+1}\rho\o\s$ with $\a<i+1$ can not be among these terms. Again we have got a contradiction.

It remains to consider  the case when $\delta([\nu^{i+1} \rho,\nu^{j}\rho])\r\s$ reduces. In this case we must have $-\a+1\leq i+1\leq \a\leq j$.
For the proof of Corollary \ref{cor1}, we need to prove two facts. The first one is that
\begin{equation*}
\delta([\nu^{-i}\rho,\nu^{c}\rho])
 \times
\delta([\nu^{j+1} \rho,\nu^{d}\rho]) \otimes L_\a(\delta([\nu^{i+1} \rho,\nu^{j}\rho]);\s)
\end{equation*}
cannot be in the Jacquet module of the $\delta([\nu^{-c} \rho,\nu^{d}\rho]_+;\s)$. Suppose, to the contrary, that it is in that Jacquet module. Now $(A6)$ (applied to $L_\a(\delta([\nu^{i+1} \rho,\nu^{j}\rho]);\s)$) shows us that in the minimal non-zero standard Jacquet module of $\delta([\nu^{-c} \rho,\nu^{d}\rho])_{+};\s)$ we would have at least one  term of the form $\dots \o \nu^{\a+1}\rho\o\s$. Recall that $(A5)$ tells us that such subquotients for the Jacquet module of $\delta([\nu^{-c} \rho,\nu^{d}\rho]_+;\s)$ can be only of the form
\begin{equation*}
\dots \o \nu^{-i+1}\rho\o\s \text{ \ where \ } \ -c+1\leq i  \leq \a,   \text{ \ or \ } \dots \o \nu^{i}\rho\o\s,  \text{ \ where \ } \ -c\leq i \leq \a.
\end{equation*}
Observe that the above terms cannot contain a term of the form $\dots \o \nu^{\a+1}\rho\o\s$. Therefore, we get a contradiction.

It can be seen in a completely analogous manner that
\begin{equation*}
\delta([\nu^{-i}\rho,\nu^{c}\rho])
 \times
\delta([\nu^{j+1} \rho,\nu^{d}\rho]) \otimes\delta([\nu^{i+1} \rho,\nu^{j}\rho]_+;\s)
\end{equation*}
cannot be in the Jacquet module of the $L_\a(\delta([\nu^{-c} \rho,\nu^{d}\rho]);\s)$.

The proof of Corollary \ref{cor1} is now complete.
\end{proof}

\begin{remark}
Suppose that $c=d<\a$. With the above convention regarding symmetrization we have
\begin{equation*}
\mu^*\left(\delta([\nu^{-c}\rho,\nu^{c}\rho]_+;\s)\right) 
 =
\sum_{i= -c
-1}^{ c}
\sum_{j=i}^{ c}
\delta([\nu^{-i}\rho,\nu^{c}\rho])
 \times
\delta([\nu^{j+1} \rho,\nu^{c}\rho]) 
\otimes
\delta([\nu^{i+1} \rho,\nu^{j}\rho]_+;\s).
\end{equation*}
Analogous relation holds if we put $L_{proper}$ (or $L_\a$) on the left-hand side, and $L_\a$ on the right-hand side (since all terms are 0).
\end{remark}

\section{Representations of segment type
corresponding  to segments   containing $[\nu^{-\a } \rho,\nu^{\a }\rho]$}

In this section we suppose $
0 \leq \a \leq c \leq d$. We shall first recall some facts from \cite{T-seg}, more details can be found there.

First we recall the case $c=d.$ Consider the representation $\delta([\nu^{-c} \rho,\nu^{c}\rho])\r\s$. It is unitarizable, multiplicity one representation of length at most two, whose each irreducible subquotient is a subrepresentation, and has $\delta([\nu^{-c} \rho,\nu^{c}\rho])\o\s$ in its Jacquet module.

For an irreducible subquotient $\pi$ of $\delta([\nu^{-c} \rho,\nu^{c}\rho])\r\s$ we easily see that
\begin{equation*}
\sum_{i=-c-1}^{-\a-1}
\delta([\nu^{-i}\rho,\nu^{c}\rho]) \t \delta([\nu^{i+1}\rho,\nu^{c}\rho])\o\s \leq s_{GL}(\pi) .
\end{equation*}

Furthermore, $\delta([\nu^{-c} \rho,\nu^{c}\rho])\r\s$ and $ \delta([\nu^{-c}\rho,\nu^{\a-1}\rho]) \rtimes \delta([\nu^{\a}\rho,\nu^{c}\rho];\s)$ have precisely one irreducible subquotient in common. This subquotient has ''the most positive part'' in its Jacquet module and is equal to $\delta([\nu^{-c} \rho,\nu^{c}\rho]_+;\s)$, while the other irreducible subquotient of $\delta([\nu^{-c} \rho,\nu^{c}\rho])\r\s$ is $\delta([\nu^{-c} \rho,\nu^{c}\rho]_-;\s).$

Let us now consider the case $ \a \leq c<d$. We have defined $\delta([\nu^{-c} \rho,\nu^{d}\rho]_+;\s)$ by the most positive term in the Jacquet module. This most positive part is  also in the Jacquet module of
\begin{equation*}
\delta([\nu^{-c}\rho,\nu^{\a-1}\rho]) \t \delta([\nu^{\a}\rho,\nu^{d}\rho])\r\s
\end{equation*}
and it follows that the multiplicity of $\delta([\nu^{-c}\rho,\nu^{d}\rho])\o\s$ is at most one in the Jacquet module of $\delta([\nu^{-c} \rho,\nu^{d}\rho]_+;\s)$. We denote the other irreducible subquotient (which is also a subrepresentation) of $\delta([\nu^{-c}\rho,\nu^{d}\rho])\r\s$ which has in its Jacquet module $\delta([\nu^{-c}\rho,\nu^{d}\rho])\o\s$ by $\delta([\nu^{-c}\rho,\nu^{d}\rho]_-;\s)$.

Since $\delta([\nu^{-d}\rho,\nu^{c}\rho])\o\s$ has the multiplicity one in the Jacquet module of $L(\delta([\nu^{-c}\rho,\nu^{d}\rho]);\s)$, this part of Jacquet module characterizes  $L(\delta([\nu^{-c}\rho,\nu^{d}\rho]);\s)$ as an irreducible subquotient of $\delta([\nu^{-c}\rho,\nu^{d}\rho])\r\s$ for which
\begin{equation*}
L(\delta([\nu^{-c}\rho,\nu^{d}\rho]);\s)\ne \delta([\nu^{-c}\rho,\nu^{d}\rho]_\pm;\s).
\end{equation*}
Therefore, we have identified three different irreducible subquotients of the induced representation $\delta([\nu^{-c}\rho,\nu^{d}\rho])\r\s$ (all of multiplicity one).

Later on we will use several times the following technical lemma:

\begin{lemma}\label{lematehn} Suppose that $\a \geq 0$ is such that $\nu^{\a} \rho \r \s$ reduces and that for $c < d$ we have $s_{GL}(L(\delta([\nu^{-c}\rho,\nu^{d}\rho]);\s)) = \sum_{i=\a}^{d} L(\delta([\nu^{-i}\rho,\nu^{c}\rho]) , \delta([\nu^{i+1}\rho,\nu^{d}\rho]))\o\s$. Let $c'\leq d'$ satisfy
\begin{equation*}
-c\leq -c'\leq-\a \leq \a\leq d'\leq d.
\end{equation*}
Then the Jacquet  module of $L(\delta([\nu^{-c}\rho,\nu^{d}\rho]);\s))$ does not contain a representation of the form
\begin{equation*}
\pi'\o\delta([\nu^{-c'}\rho,\nu^{d'}\rho])\o\s
\end{equation*}
for any irreducible representation $\pi'$ of a general linear group.

\end{lemma}

\begin{proof}  Suppose that the above term is contained in  Jacquet module of $L(\delta([\nu^{-c} \rho,\nu^{d}\rho]);\s)$. Obviously, then $c<d$ (otherwise $L(\delta([\nu^{-c} \rho,\nu^{d}\rho]);\s)=0$). Now the assumption on $s_{GL}(L(\delta([\nu^{-c}\rho,\nu^{d}\rho])$;$\s))$ implies that the  representation $\pi'\o\delta([\nu^{-c'}\rho,\nu^{d'}\rho])\o\s$ appears in  the Jacquet module of
\begin{equation*}
L(\delta([\nu^{-i'}\rho,\nu^{c}\rho]) ,
\delta([\nu^{i'+1} \rho,\nu^{d}\rho]))\o\s,  \quad \text{for some} \quad \a\leq i'\leq d.
\end{equation*}
Using the formula for $m^*$ we see that, to be able to get $\pi' \o \delta([\nu^{-c'} \rho,\nu^{d'}\rho])\o\s$ in Jacquet module of the above representation, we must have $-c'=i'+1$ or $-c'=-i'$. Obviously, the first relation cannot hold (because of different signs). Therefore, $i'=c'$  and it remains to consider $m^*(L(\delta([\nu^{-c'}\rho,\nu^{c}\rho]) $,
$\delta([\nu^{c'+1} \rho,\nu^{d}\rho]))).$ But, by \cite{KL},
\begin{equation*}
m^*(L(\delta([\nu^{-c'}\rho,\nu^{c}\rho]) , \delta([\nu^{c'+1} \rho,\nu^{d}\rho])))
\end{equation*}
does not contain a term of the form $\pi' \o \delta([\nu^{-c'} \rho,\nu^{d'}\rho])$ with $d' \geq c'$.

This proves that the (non-zero) terms of the form $ \ldots\o \delta([\nu^{-c'} \rho,\nu^{d'}\rho]\o \s$ can not be in the Jacquet module of Langlands, proving the lemma.
\end{proof}

The following relations have been obtained in \cite{T-seg}:
\begin{align*}
\sum_{i=-c-1}^{-\alpha-1}
\delta([\nu^{i+1}\rho,\nu^{d}\rho])\times
\delta([\nu^{-i}\rho,\nu^{c}\rho]) \otimes \sigma & \\
& \leq
s_{GL}(\delta([\nu^{-c}\rho,\nu^{d}\rho]_+,
\sigma)) \\
& \leq
  \sum_{i=-c-1}^{\alpha-1}
\delta([\nu^{i+1}\rho,\nu^{d}\rho])\times
\delta([\nu^{-i}\rho,\nu^{c}\rho]) \otimes \sigma
\end{align*}
and
\begin{equation*}
s_{GL}(\delta([\nu^{-c}\rho,\nu^{d}\rho]_-,
\sigma))
=
\sum_{i=-c-1}^{-\a-1}
\delta([\nu^{i+1}\rho,\nu^{d}\rho])\times
\delta([\nu^{-i}\rho,\nu^{c}\rho]) \otimes \sigma.
\end{equation*}

The following theorem will give us further details about the Jacquet modules of irreducible subquotients of $\delta([\nu^{-c} \rho,\nu^{d}\rho])\r\s$.

To keep the notation of our results uniform, we continue with our convention regarding meaning of $\delta([\nu^{-c}\rho,\nu^{d}\rho]_+;\s)$ and $L_\a(\delta([\nu^{-c}\rho,\nu^{d}\rho]);\s)$ when $\delta([\nu^{-c}\rho,\nu^{d}\rho])\rtimes \s$ is irreducible.

\begin{theorem} \label{teoremdrugi} Let $c,d\in(1/2)\Z$ such that $d+c,d-c\in\Z_{\geq 0}$, $d-\a\in\Z$ and
$$
-c\leq-\a< \a\leq d.
$$
Then
\begin{enumerate}

\item[$(B1)$] If $c<d$ then $\delta([\nu^{-c}\rho,\nu^{d}\rho])\r\s$ is a representation of length three, and the composition series consists of two subrepresentations: $\delta([\nu^{-c}\rho,\nu^{d}\rho]_\pm;\s)$ and the Langlands quotient $L(\delta([\nu^{-c}\rho,\nu^{d}\rho]);\s)$. \\ For $c=d$,  $\delta([\nu^{-c}\rho,\nu^{c}\rho])\r\s$ is a representation of length two, and the composition series consists of $\delta([\nu^{-c}\rho,\nu^{c}\rho]_\pm;\s)$, which are both subrepresentations.

\item[$(B2)$] If $c<d$ then $\delta([\nu^{-c}\rho,\nu^{d}\rho]_\pm;\s)$ are square integrable. For $c=d$, the representations are tempered, but not square integrable.

\item[$(B3)$] $
s_{GL}(\delta([\nu^{-c}\rho,\nu^{d}\rho]_\pm;\s))=\sum_{i=-c-1}^{\pm\a-1}
\delta([\nu^{-i}\rho,\nu^{c}\rho]) \t \delta([\nu^{i+1}\rho,\nu^{d}\rho])\o\s.
$

\item[$(B4)$] For $c<d$

 $
s_{GL}(L(\delta([\nu^{-c}\rho,\nu^{d}\rho]);\s))
=
\sum_{i=\a}^{d}
L(\delta([\nu^{-i}\rho,\nu^{c}\rho]) , \delta([\nu^{i+1}\rho,\nu^{d}\rho]))\o\s.
$

\item[$(B5)$] $s_{top} (\delta([\nu^{-c}\rho,\nu^{c}\rho]_+;\s)) =  2\, \nu^{c}\rho\o\delta([\nu^{-c+1}\rho,\nu^{c}\rho]_+;\s) + \nu^{c}\rho\o L_\a(\delta([\nu^{-c+1}\rho,\nu^{
c}\rho]);\s),$
$s_{top} (\delta([\nu^{-c}\rho,\nu^{c}\rho]_-;\s)) = 2\, \nu^{c}\rho\o\delta([\nu^{-c+1}\rho,\nu^{c}\rho]_-;\s) +  \nu^{c}\rho\o L_\a(\delta([\nu^{-c+1}\rho,\nu^{c}\rho]);\s).$

\item[$(B6)$] For $c<d$
\begin{align*}
s_{top} (\delta([\nu^{-c}\rho,\nu^{d}\rho]_\pm;\s)) & = \nu^{d}\rho\o\delta([\nu^{-c}\rho,\nu^{d-1}\rho]_\pm;\s) +{} \\
& \qquad + \nu^{c}\rho\o\delta([\nu^{-c+1}\rho,\nu^{d}\rho]_\pm;\s).
\end{align*}

\item[$(B7)$] For $c<d$
\begin{align*}
s_{top}(L(\delta([\nu^{-c}\rho,\nu^{d}\rho]);\s)) & = \nu^{d}\rho\o L_{\a}(\delta([\nu^{-c}\rho,\nu^{d-1}\rho]);\s) +{} \\
& \qquad + \nu^{c}\rho\o L_{\a}(\delta([\nu^{-c+1}\rho,\nu^{d}\rho]);\s).
\end{align*}

\end{enumerate}
\end{theorem}

\begin{proof}
We emphasize that it is easy to see, using formula for $\mu^{\ast}$ and well-known composition series of induced representations of general linear groups, that the sum of the two sums on the right-hand side of $(B3)$ and the sum on  the right-hand side of $(B4)$ equals $s_{GL}(\delta([\nu^{-c} \rho,\nu^{d}\rho])\r\s$).

Observe that the statements in $(B2)$ regarding square integrability and temperedness of involved representations follow from \cite{T-seg}. From there we also know the length two claim and the statement $(B3)$ in the tempered case. Furthermore, we know that claims $(B3)$ and $(B6)$ hold for $\delta([\nu^{-c}\rho,\nu^{d}\rho]_-;\s)$. Also, the length three claim in $(B1)$ now follows from Theorem 2.1 of \cite{Mu-CSSP}, while the length two claim in $(B1)$ is an integral part of classification of discrete series (\cite{Moe-T}).

We shall prove the rest of the theorem by interlaced inductions on $c$ and $d$.

The formula for $\mu^*$ implies
\begin{equation*}
s_{top} (\delta([\nu^{-c}\rho,\nu^{d}\rho])\r\s)=
\nu^{d}\rho\o\delta([\nu^{-c}\rho,\nu^{d-1}\rho])\r\s + \nu^{c}\rho\o\delta([\nu^{-c+1}\rho,\nu^{d}\rho])\r\s.
\end{equation*}

First we shall prove the theorem in case $c=\a$. The proof goes by induction on $d$.

First consider the case $d=\a$. We know that in this case $\delta([\nu^{-\a}\rho,\nu^{\a}\rho])\r\s$ is of length two. The above formula and the previous theorem, together with symmetrization of notation, imply that
\begin{equation*}
s_{top} (\delta([\nu^{-\a}\rho,\nu^{\a}\rho])\r\s)= 2\, \nu^{\a}\rho\o\delta([\nu^{-\a+1}\rho,\nu^{\a}\rho]_+;\s) + 2\, \nu^{\a}\rho\o L(\delta([\nu^{-\a+1}\rho,\nu^{\a}\rho]);\s)
\end{equation*}
is the decomposition of $s_{top} (\delta([\nu^{-\a}\rho,\nu^{\a}\rho])\r\s)$ into irreducible representations. From the definition of $\delta([\nu^{-\a}\rho,\nu^{\a}\rho]_+;\s)$ we know that $2\, \nu^{\a}\rho\o\delta([\nu^{-\a+1}\rho,\nu^{\a}\rho]_+;$ $\s)$ must belong to its Jacquet module. On the other hand, it follows from $(B3)$, which, as we know, does hold in this case, that the minimal non-zero Jacquet module of $\delta([\nu^{-\a}\rho,\nu^{\a}\rho]_-;\s)$ is irreducible. Now from these two facts it follows that
\begin{align*}
s_{top} (\delta([\nu^{-\a}\rho,\nu^{\a}\rho]_+;\s)) &=
2\nu^{\a}\rho\o\delta([\nu^{-\a+1}\rho,\nu^{\a}\rho]_+;\s) + \nu^{\a}\rho\o L(\delta([\nu^{-\a+1}\rho,\nu^{\a}\rho]);\s), \\
s_{top} (\delta([\nu^{-\a}\rho,\nu^{\a}\rho]_-;\s)) &=  \nu^{\a}\rho\o L(\delta([\nu^{-\a+1}\rho,\nu^{\a}\rho]);\s).
\end{align*}
Since $\delta([\nu^{-\a+1}\rho,\nu^{\a}\rho]_-;\s)=0$, this is exactly $(B5)$ in this case, and it implies that the theorem holds for $c = d = \a$ (note that $(B6)$ and $(B7)$ do not apply here).

Fix $d>\a$ and assume that the theorem holds for $d-1$. Then the inductive assumption implies the following decomposition into irreducible representations:
\begin{multline*}
s_{top} (\delta([\nu^{-\a}\rho,\nu^{d}\rho])\r\s)
=\nu^{d}\rho\o\delta([\nu^{-\a}\rho,\nu^{d-1}\rho]_+;\s) + \nu^{\a}\rho\o\delta([\nu^{-\a+1}\rho,\nu^{d}\rho]_+;\s) \\
+\nu^{d}\rho\o\delta([\nu^{-\a}\rho,\nu^{d-1}\rho]_-;\s)
+\nu^{d}\rho\o L(\delta([\nu^{-\a}\rho,\nu^{d-1}\rho]);\s) + \nu^{\a}\rho\o L(\delta([\nu^{-\a+1}\rho,\nu^{d}\rho]);\s)
\end{multline*}
(note that, according to our notation, in this case $\nu^{\a}\rho\o\delta([\nu^{-\a+1}\rho,\nu^{d}\rho]_-;\s)=0$).
The first two terms obviously belong to the Jacquet module of $\delta([\nu^{-\a}\rho,\nu^{d}\rho]_+;\s)$.
Since the representation $\delta([\nu^{-\a}\rho,\nu^{d}\rho]_-;\s)$ has in its minimal non-zero Jacquet module the term $\nu^d\rho\o\nu^{d-1}\rho\o\dots\o\nu^{-\a}\rho\o\s$, which cannot come from the last two terms, the third term must come from  $\delta([\nu^{-\a}\rho,\nu^{d}\rho]_-;\s)$.

The properties of Langlands classification imply that in the Jacquet module of Langlands quotient $L(\delta([\nu^{-\a}\rho,\nu^{d}\rho]);\s)$ we must have a term of the form $\nu^{\a}\rho\o\dots$ Therefore, the last term must be in this Jacquet module. If $d=\a+1$, the fourth term is zero. If not, then the Langlands quotient embeds into
\begin{align*}
\delta([\nu^{-d}\rho,\nu^{\a}\rho])\r\s \h \nu^{\a}\rho \t \cdots \t \nu^{-d}\rho\r\s & \cong  \nu^{\a}\rho \t \cdots \t \nu^{d}\rho\r\s \\
& \cong \nu^{d}\rho\t   \nu^{\a}\rho \t \cdots \t \nu^{-d+1}\rho\r\s.
\end{align*}
Consequently, Jacquet module of the Langlands quotient has to contain a representation of the form $\nu^{d}\rho\o\ldots $. Since $\a+1\ne d$, we see that the fourth term must be in the Jacquet module of Langlands quotient. This proves, in this case, the claimed formulas for the top Jacquet modules.

We have already seen that the right-hand side of $(B3)$ presents an upper bound for $s_{GL}(\delta([\nu^{-\a}\rho,\nu^{d}\rho]_+;\s))$.
Applying Jacquet modules $r_{(p,(d+c)p)}$ to this upper bound and on $(B6)$ for $\delta([\nu^{-\a}\rho,\nu^{d}\rho]_+;\s)$ (which we have just proved), in the same way as in the proof of Theorem \ref{tmprvi}, we get an equality which proves $(B3)$. But this also implies $(B4)$, because we have
\begin{equation*}
s_{GL}(\delta([\nu^{-c}\rho,\nu^{d}\rho])\r\s) =
\sum_{i=-c-1}^{d}
\delta([\nu^{-i}\rho,\nu^{c}\rho]) \t \delta([\nu^{i+1}\rho,\nu^{d}\rho])\o\s,
\end{equation*}
since $\delta([\nu^{-c}\rho,\nu^{d}\rho])\r\s$ is a representation of length three, and we know Jacquet modules of the remaining two irreducible subquotients.

This completes the proof for $c=\a$.

Now we shall fix $c>\a$ and assume that the theorem holds for $c-1$. Again, we proceed inductively, similarly as in case $c=\a$.

Let us start with the case $d=c$.
We know that in this case $\delta([\nu^{-c}\rho,\nu^{c}\rho])\r\s$ is of length two. The above formula and symmetrization of notation imply that
\begin{multline*}
s_{top} (\delta([\nu^{-c}\rho,\nu^{c}\rho])\r\s)= 2\, \nu^{c}\rho\o\delta([\nu^{-c+1}\rho,\nu^{c}\rho]_+;\s) \\
+2\,\nu^{c}\rho\o\delta([\nu^{-c+1}\rho,\nu^{c}\rho]_-;\s) + 2\, \nu^{c}\rho\o L(\delta([\nu^{-c+1}\rho,\nu^{c}\rho]);\s)
\end{multline*}
is a decomposition of $s_{top} (\delta([\nu^{-c}\rho,\nu^{c}\rho])\r\s)$ into irreducible representations. Definition of the representation $\delta([\nu^{-c}\rho,\nu^{c}\rho]_+;\s)$ directly implies that the irreducible representation $2\, \nu^{c}\rho\o\delta([\nu^{-c+1}\rho, \nu^{c}\rho]_+;\s)$ belongs to its Jacquet module. Furthermore, since $\delta([\nu^{-c}\rho,\nu^{c}\rho])\o\s$ is in the Jacquet module of both $\delta([\nu^{-c}\rho,\nu^{c}\rho]_+;\s)$ and $\delta([\nu^{-c}\rho,\nu^{c}\rho]_-;\s)$, using transitivity of Jacquet modules and the fact that $\nu^{-c}\rho$ does not appear in the cuspidal support of $s_{GL}(\delta([\nu^{-c+1}\rho,\nu^{c}\rho]_+;\s))$, we obtain that $ \nu^{c}\rho\o L(\delta([\nu^{-c+1}\rho,\nu^{c}\rho]);\s)$ must be in both Jacquet modules. From the formula for $s_{GL}$ of  $\delta([\nu^{-c}\rho,\nu^{c}\rho]_\pm;\s)$ it follows directly that the multiplicity of $\nu^{c}\rho \o\nu^{c}\rho \o \nu^{c-1}\rho \o \dots \o \nu^{-c+1}\rho\o\s$ is the same for both representations (since $c>\a$). Observe that $\nu^{c}\rho \o\nu^{c}\rho \o \nu^{c-1}\rho \o \dots \o \nu^{-c+1}\rho\o\s$ has positive multiplicity in $\nu^{c}\rho\o\delta([\nu^{-c+1}\rho,\nu^{c}\rho]_-;\s) $ and multiplicity zero in $\nu^c\rho\o L(\delta([\nu^{-c+1}\rho,\nu^{c}\rho]);\s)$. It follows immediately that $2\, \nu^{c}\rho\o\delta([\nu^{-c+1}\rho,\nu^{c}\rho]_-;\s) $ appears in the Jacquet module of $\delta([\nu^{-c}\rho,\nu^{c}\rho]_-;\s)$. This completes the proof of the theorem in this case.

Let us now fix $d>c$ and assume that the theorem holds for $d-1$ and $c$, and also for $c-1$ and all $d\geq c-1$. The inductive assumptions give the following decomposition into irreducible representations
\begin{multline*}
s_{top} (\delta([\nu^{-c}\rho,\nu^{d}\rho])\r\s)= \nu^{d}\rho\o\delta([\nu^{-c}\rho,\nu^{d-1}\rho]_+;\s) + \nu^{c}\rho\o\delta([\nu^{-c+1}\rho,\nu^{d}\rho]_+;\s) \\
+\nu^{d}\rho\o\delta([\nu^{-c}\rho,\nu^{d-1}\rho]_-;\s) + \nu^{c}\rho\o\delta([\nu^{-c+1}\rho,\nu^{d}\rho]_-;\s) \\
+\nu^{d}\rho\o L(\delta([\nu^{-c}\rho,\nu^{d-1}\rho]);\s) + \nu^{c}\rho\o L(\delta([\nu^{-c+1}\rho,\nu^{d}\rho]);\s).
\end{multline*}
In the same way as in the case $c = \a$ we deduce that the first two terms belong to the Jacquet module of $\delta([\nu^{-c}\rho,\nu^{d}\rho]_+;\s)$, while
$s_{top}(\delta([\nu^{-\a}\rho,\nu^{d}\rho]_-;\s))$ contains the third one. Also, since $\nu^{-d}\rho$ appears only in the last term, that term belongs to Jacquet module of Langlands quotient.

If $d=c+1$, then the fifth  representation is zero. If not, i.e., if $d>c+1$, in the same way as before we see that the fifth representation is also in the Jacquet module of the Langlands quotient.

Observe that the Jacquet module of $GL$-type of $\delta([\nu^{-c}\rho,\nu^{d}\rho]_-;\s)$ contains $\delta([\nu^{-c+1}\rho,\nu^{d}\rho])\t\nu^c\rho\o\s$. This implies that in the Jacquet module of this representations we must have subquotients of the form $\nu^c\rho\o\dots$. Therefore, the fourth representation is in the Jacquet module of $\delta([\nu^{-\a}\rho,\nu^{d}\rho]_-;\s)$. This proves the formulas for the top Jacquet modules as claimed in the theorem.

In completely analogous manner as in the case $c = \a$, we prove $(B3)$ for representation $\delta([\nu^{-\a}\rho,\nu^{d}\rho]_+;\s)$ by showing equality on the level of Jacquet modules $r_{(p,(d+c)p)}$ (applying this Jacquet modules to the right-hand side of $(B3)$, which we know is an upper bound for $s_{GL}(\delta([\nu^{-\a}\rho,\nu^{d}\rho]_+;\s))$, and on $(B6)$). This also implies $(B4)$ and completes the proof of Theorem \ref{teoremdrugi}.
\end{proof}

We take a moment to provide an interpretation of the results obtained in the previous theorem in terms of admissible triples.

\begin{remark}
If $-c\leq-\a< \a\leq d$ and $\delta([\nu^{-c}\rho,\nu^{d}\rho]_\pm;\s)$ are square integrable (i.e., $c < d$) then $\Jord(\delta([\nu^{-c}\rho,\nu^{d}\rho]_\pm;\s)) = \Jord(\s) \cup \{ (2c+1, \rho), (2d+1, \rho) \}$. In addition, if we denote the $\e$-function corresponding to $\delta([\nu^{-c} \rho,\nu^{d}\rho]_+;\s)$ (resp., $\delta([\nu^{-c}\rho,\nu^{d}\rho]_-;\s)$) by $\e_+$ (resp., $\e_-$), then we obviously have $d\_ = c$ and
$ \e_\pm ((c,\rho), (d, \rho))= 1$. Furthermore, by $(B3)$ we have $\e_+ ((c\_, \rho), (c, \rho))= 1$, and $\e_- ((c\_, \rho), (c, \rho))= -1$ if $c\_$ exists and $\e_{\pm}(c, \rho) = \pm 1$ otherwise.
\end{remark}

In the following corollary we determine all Jacquet modules for representations of segment type in the cases considered.

\begin{corollary} \label{korprvi} Let $c,d\in(1/2)\Z$ be such that $d+c,d-c\in\Z_{\geq 0}$, $d-\a\in\Z$ and $-c\leq-\a< \a\leq d$. Then
\begin{align*}
\mu^*&\left(\delta([\nu^{-c}\rho,\nu^{d}\rho]_\pm;\s)\right) 
\\
& = \sum_{i = -c-1}^{d-1%
} \sum_{j = i+1}^{d}
\delta([\nu^{-i}\rho,\nu^{c}\rho])
 \times
\delta([\nu^{j+1} \rho,\nu^{d}\rho]) \otimes
\delta([\nu^{i+1} \rho,\nu^{j}\rho]_\pm;\s) + 
\\
& + \mkern-20mu\sum_{-c-1\le i\le c-1%
}\ \sum_{i+1\le j\le
c}\mkern-75mu\rule[-4.5ex]{0pt}{2ex}_{i+j < -1}\mkern25mu
\delta([\nu^{-i}\rho,\nu^{c}\rho])
 \times
\delta([\nu^{j+1} \rho,\nu^{d}\rho]) \otimes
L_\a(\delta([\nu^{i+1} \rho,\nu^{j}\rho]);\s) + 
\\[-1.5ex]
&\mkern100mu+ \mkern-10mu\sum_{i=-c-1}^{\pm \a - 1}
\delta([\nu^{-i}\rho,\nu^{c}\rho]) \t \delta([\nu^{i+1}\rho,\nu^{d}\rho])\o\s.
\end{align*}
For $c<d$ we have \begin{align*}
\mu^*\big(&L(\delta([\nu^{-c}\rho,\nu^d\rho]);\sigma)\big)  
\\
&= \mu^\ast\big(L_\alpha(\delta([\nu^{-c}\rho,\nu^d\rho]);\sigma)\big) = \mu^\ast\big(L_{proper}(\delta([\nu^{-c}\rho,\nu^d\rho]);\sigma)\big)  
\\
&= \mkern-20mu\sum_{-c-1\le i\le d-1%
}\ \sum_{i+1\le j\le
d}\mkern-75mu\rule[-4.5ex]{0pt}{2ex}_{0\le i+j}\mkern25mu
L(\delta([\nu^{-i}\rho,\nu^c\rho]),\delta([\nu^{j+1}\rho,\nu^d\rho])\big) \otimes
L_\alpha(\delta([\nu^{i+1}\rho,\nu^j\rho]);\sigma)) +{} 
 \\[-1.5ex]
&\mkern100mu+ \mkern+1mu \sum_{i=\a}^d
L(\delta([\nu^{-i}\rho,\nu^c\rho]),\delta([\nu^{i+1}\rho, \nu^d\rho])) \otimes
\sigma. 
\end{align*}
\end{corollary}

\begin{proof} We have
\begin{align*}
\mu^*\left(\delta([\nu^{-c}\rho,\nu^{d}\rho])\r\s\right) & =
\sum_{i= -c
-1}^{ d}
\sum_{j=i}^{ d}
\big ( \delta([\nu^{-i}\rho,\nu^{c}\rho])
 \times
\delta([\nu^{j+1} \rho,\nu^{d}\rho]) \big ) \otimes{}
\\
& \qquad \otimes
\big(\delta([\nu^{i+1} \rho,\nu^{j}\rho]_+;\s)+ \delta([\nu^{i+1} \rho,\nu^{j}\rho]_-;\s) +{} \\
& \qquad + L_\a(\delta([\nu^{i+1} \rho,\nu^{j}\rho]);\s)\big).
\end{align*}

In the previous theorem  we have determined the Jacquet modules of $GL$-type which coincide with the appropriate terms in the above formulas. Therefore, it remains to consider the case $i<j$.

We shall first analyze summands of the form $\ldots \o \delta([\nu^{i+1} \rho,\nu^{j}\rho]_+;\s)$ when this representation is non-zero (i.e., when $[i+1,j]\cap [-\a,\a]\ne\emptyset$). Then Theorems \ref{tmprvi} and \ref{teoremdrugi} imply that the minimal non-zero standard Jacquet module of this representation  contains at least one  term of the form $\dots \o \nu^{\a}\rho\o\s$.

By transitivity of Jacquet modules, if $\pi$ is an irreducible subquotient of $\delta([\nu^{-c} \rho,\nu^{d}\rho]) \rtimes \s$ such that $\mu^{\ast}(\pi) \geq \ldots \o \delta([\nu^{i+1} \rho,\nu^{j}\rho]_+;\s)$, then its minimal non-zero standard Jacquet module contains a term of the form $\dots \o \nu^{\a}\rho\o\s$. Using Theorems \ref{tmprvi} and \ref{teoremdrugi}, we see that minimal non-zero standard Jacquet modules of $\delta([\nu^{-c} \rho,\nu^{d}\rho]_-;\s)$ and of $L(\delta([\nu^{-c} \rho,\nu^{d}\rho]);\s)$  do not contain a representation of the form $\ldots \o \nu^{\a}\rho\o\s$. Consequently, all terms of the form $\ldots \o \delta([\nu^{i+1} \rho,\nu^{j}\rho]_+;\s)$ are in the Jacquet module of $\delta([\nu^{-c} \rho,\nu^{d}\rho]_+;\s)$.

Let us now analyze the case of summands
$$
\delta([\nu^{-i}\rho,\nu^{c}\rho])
 \times
\delta([\nu^{j+1} \rho,\nu^{d}\rho]) \o \delta([\nu^{i+1} \rho,\nu^{j}\rho]_-;\s)
$$
when this representation is non-zero,
i.e., when $[-\a,\a]\subseteq [i+1,j]$.
Then $i\leq -\a-1$ and $\a\leq j$. Now Lemma \ref{lematehn} and Theorem \ref{teoremdrugi} imply that there are no irreducible subquotients of this term in Jacquet module of the Langlands quotient.
Suppose that some subquotient of the form $\pi'\o  \delta([\nu^{i+1} \rho,\nu^{j}\rho]_-;\s)$ appears in the Jacquet module of $\delta([\nu^{-c} \rho,\nu^{d}\rho]_+;\s)$. Two possibilities will be studied separately.

We first consider the case $|i+1|\leq j$, and let $\varphi$ stand for any irreducible subquotient of the minimal non-zero Jacquet module of $\pi'\o \delta([\nu^{i+1} \rho,\nu^{j}\rho])\o\s$. Then $\varphi $ must also be in the Jacquet module of $\delta([\nu^{-c} \rho,\nu^{d}\rho]_+;\s)$. By Theorem \ref{teoremdrugi}, $s_{GL}(\delta([\nu^{-c} \rho,\nu^{d}\rho]_-;\s))\leq s_{GL}(\delta([\nu^{-c} \rho,\nu^{d}\rho]_+;\s)$, and the difference is $\sum_{i=-\a}^{\a-1}
\delta([\nu^{-i}\rho,\nu^{c}\rho]) \t \delta([\nu^{i+1}\rho,\nu^{d}\rho])\o\s$. Passing to the minimal non-zero Jacquet module, this will give terms of the form
$$
\ldots\o\nu^{-\a+1}\rho\o\s, \ldots\o\nu^{-\a+2}\rho\o\s, \ldots, \ldots\o\nu^{\a-1}\rho\o\s, \ldots\o\nu^{\a}\rho\o\s.
$$
Since $i+1\leq -\a$, we see that $\varphi$ cannot be in Jacquet module of the difference. Therefore, $\varphi$ is also contained in the Jacquet module of
$\delta([\nu^{-c} \rho,\nu^{d}\rho]_-;\s)$. In the same way we see that the multiplicity of $\varphi$ in Jacquet modules of both representations $\delta([\nu^{-c} \rho,\nu^{d}\rho]_\pm;\s)$ is the same (and strictly positive).

Now we shall study the multiplicity of $\varphi$ in Jacquet modules of $\delta([\nu^{-c} \rho,\nu^{d}\rho]_\pm; \s)$ using transitivity of Jacquet modules, through the parabolic subgroup corresponding to $\pi'\o \delta([\nu^{i+1} \rho,\nu^{j}\rho])\o\s$. For $c<j$, we clearly need to study only
\begin{multline*}
\delta([\nu^{-i}\rho,\nu^{c}\rho])
 \times
\delta([\nu^{j+1} \rho,\nu^{d}\rho]) \otimes
\delta([\nu^{i+1} \rho,\nu^{j}\rho]_+;\s) + {} \\
\delta([\nu^{-i}\rho,\nu^{c}\rho])
 \times
\delta([\nu^{j+1} \rho,\nu^{d}\rho]) \otimes
\delta([\nu^{i+1} \rho,\nu^{j}\rho]_-;\s).
\end{multline*}
Applying Theorem \ref{teoremdrugi} to $\delta([\nu^{i+1} \rho,\nu^{j}\rho]_\pm;\s)$, we get that the multiplicity of $\varphi$ in the first summand is greater than or equal to its multiplicity in the second summand. We have already proved that the first term belongs entirely to the Jacquet module of $\delta([\nu^{-c} \rho,\nu^{d}\rho]_+;\s)$. Now our assumption, that the subquotient $\pi'\o\delta([\nu^{i+1} \rho,\nu^{j}\rho]_-;\s)$ of the second sum, in which $\varphi$ has positive multiplicity, belongs to Jacquet module of $\delta([\nu^{-c} \rho,\nu^{d}\rho]_+;\s)$, implies that the multiplicity of $\varphi$ in $\delta([\nu^{-c} \rho,\nu^{d}\rho]_+;\s)$ is strictly greater then the multiplicity of $\varphi$ in Jacquet module of $\delta([\nu^{-c} \rho,\nu^{d}\rho]_-;\s)$. This contradicts the fact that multiplicities are the same.

The case $j\leq c$ can be handled in the same way, but more easily.

This completes the proof that all terms of the form $\pi'\o\delta([\nu^{i+1} \rho,\nu^{j}\rho]_-;\s)$ are in the Jacquet module of $\delta([\nu^{-c} \rho,\nu^{d}\rho]_-;\s)$ if $|i+1| \leq j$. The case $j<|i+1|$ can be handled in completely analogous way, but working with the segment $[\nu^{-j} \rho,\nu^{-i-1}\rho]$ instead of $[\nu^{i+1} \rho,\nu^{j}\rho]$.

What is left is to determine where do the non-zero representations of the form $\pi'\o L( \delta([\nu^{i+1} \rho$,$\nu^{j}\rho]);\s)$ belong. Clearly, we can assume $[i+1,j]\not\subseteq [-\a+1,\a-1]$.

First observe that if $c+1\leq j$, then $\pi'\o L( \delta([\nu^{i+1} \rho,\nu^{j}\rho]);\s)$ contains a representation of the form $\nu^{-j}\rho$ in the cuspidal support. By the  previous theorem, is not in the discrete series, and in this case the above term belongs to the Jacquet module of $L( \delta([\nu^{-c} \rho,\nu^{d}\rho]);\s)$.

It remains to consider the case $j\leq c$. Let us first assume $i+j<-1$. Since $i+1\leq j$, we obtain $|j|\leq -i-1$. Now the condition $L( \delta([\nu^{i+1} \rho,\nu^{j}\rho]);\s) \neq 0$ gives $\a\leq -i-1$, i.e., $i+1\leq - \a$. By Theorem \ref{teoremdrugi}, $\d([\nu^{i+1}\rho,\nu^{d}\rho])\t\d([\nu^{-i}\rho,\nu^{c}\rho])\o\s$ appears in the Jacquet module of $\delta([\nu^{-c} \rho,\nu^{d}\rho]_\pm;\s)$. Looking at
\begin{multline*}
m^* \big (\d([\nu^{i+1}\rho,\nu^{d}\rho])\t\d([\nu^{-i}\rho,\nu^{c}\rho]) \big ) = \\
\big (\sum_{l = i}^{d} \d([\nu^{l+1}\rho,\nu^{d}\rho]) \otimes \d([\nu^{i+1}\rho,\nu^{l}\rho]) \big ) \times \big ( \!\!\! \sum_{k = -i-1}^{c} \d([\nu^{k+1}\rho,\nu^{c}\rho]) \otimes \d([\nu^{-i}\rho,\nu^{k}\rho]) \big )
\end{multline*}
we deduce that Jacquet modules of both $\delta([\nu^{-c} \rho,\nu^{d}\rho]_\pm;\sigma)$ contain
\begin{equation*}
\delta([\nu^{j+1}\rho,\nu^{d}\rho]) \times \delta([\nu^{-i}\rho,\nu^{c}\rho]) \otimes \delta([\nu^{i+1}\rho,\nu^{j}\rho]) \otimes \s.
\end{equation*}
Therefore,
\begin{multline*}
\delta([\nu^{j+1}\rho,\nu^{d}\rho])
 \times
\delta([\nu^{-i} \rho,\nu^{c}\rho]) \otimes
L(\delta([\nu^{-j} \rho,\nu^{-i-1}\rho]);\sigma) = \\
\delta([\nu^{j+1}\rho,\nu^{d}\rho])
 \times
\delta([\nu^{-i} \rho,\nu^{c}\rho]) \otimes
L(\delta([\nu^{i+1} \rho,\nu^{j}\rho]);\sigma)
\end{multline*}
appears in Jacquet modules of both $\delta([\nu^{-c} \rho,\nu^{d}\rho]_\pm;\sigma)$ and appears with multiplicity two in Jacquet module of $\delta([\nu^{-c}\rho,\nu^{d}\rho]) \rtimes \s$.

Note that for $i + j = -1$ we have either $\delta([\nu^{i+1} \rho,\nu^{j}\rho]) \rtimes \s = \delta([\nu^{i+1} \rho,\nu^{j}\rho]_+;\sigma) + \delta([\nu^{i+1} \rho,\nu^{j}\rho]_-;\sigma)$ or $\delta([\nu^{i+1} \rho,\nu^{j}\rho]) \rtimes \s = \delta([\nu^{i+1} \rho,\nu^{j}\rho]_+;\sigma)$.

Now assume $-1<i+j$.

We immediately get $|i+1|\leq j$, while the non-triviality of $L(\delta([\nu^{i+1} \rho,\nu^{j}\rho]);\sigma)$ implies $\a\leq j$.
By Theorem \ref{teoremdrugi}, Jacquet module of the Langlands quotient
$
L(\delta([\nu^{-c} \rho,\nu^{d}\rho]);
\sigma)
$
contains $L(\delta([\nu^{-j}\rho,\nu^{c}\rho]) ,
\delta([\nu^{j+1} \rho,\nu^{d}\rho])) \otimes
\sigma$. Consequently, in the Jacquet module of $L(\delta([\nu^{-c} \rho,\nu^{d}\rho]); \sigma)$,
$$m^* \big (L(\d([\nu^{-j}\rho,\nu^{c}\rho]), \d([\nu^{j+1}\rho,\nu^{d}\rho])) \big ) \otimes \s,$$
must also appear, and equals
$$\big (\sum_{l = j}^{d} \d([\nu^{l+1}\rho,\nu^{d}\rho]) \otimes \d([\nu^{j+1}\rho,\nu^{l}\rho]) \big ) \times \big ( \sum_{k = -j-1}^{c} \d([\nu^{k+1}\rho,\nu^{c}\rho]) \otimes \d([\nu^{-j}\rho,\nu^{k}\rho]) \big ) \otimes \s$$
$$- \big (\sum_{l' = -j-1}^{d} \d([\nu^{l'+1}\rho,\nu^{d}\rho]) \otimes \d([\nu^{-j}\rho,\nu^{l'}\rho]) \big ) \times \big ( \sum_{k' = j}^{c} \d([\nu^{k'+1}\rho,\nu^{c}\rho]) \otimes \d([\nu^{j+1}\rho,\nu^{k'}\rho]) \big ) \otimes \s. $$
For pairs of indices $(k, l) = (-i-1,j)$ and $(k', l') = (j,-i-1)$ we get that
\begin{equation*}
L(\d([\nu^{-i}\rho,\nu^{c}\rho]), \d([\nu^{j+1}\rho,\nu^{d}\rho])) \otimes \d([\nu^{-j}\rho,\nu^{-i-1}\rho]) \otimes \s
\end{equation*}
is contained in Jacquet module of $L(\delta([\nu^{-c} \rho,\nu^{d}\rho]);
\sigma)$
and it directly follows that
$$
L(\delta([\nu^{-i}\rho,\nu^{c}\rho]) ,
\delta([\nu^{j+1} \rho,\nu^{d}\rho])) \otimes
L(\delta([\nu^{i+1} \rho,\nu^{j}\rho]);\sigma)
$$
is contained in the Jacquet module of
$L(\delta([\nu^{-c} \rho,\nu^{d}\rho]);
\sigma)$.

An irreducible constituent of the form $\pi' \otimes L(\delta([\nu^{i+1} \rho,\nu^{j}\rho]);\sigma)$, where $i < j$, appears in $\mu^{\ast}(\delta([\nu^{-c} \rho,\nu^{d}\rho]) \rtimes
\sigma)$ either as a subquotient of
\begin{equation*}
\delta([\nu^{-i} \rho,\nu^{c}\rho]) \times \delta([\nu^{j+1} \rho,\nu^{d}\rho]) \otimes  \delta([\nu^{i+1} \rho,\nu^{j}\rho]) \rtimes \s
\end{equation*}
or as a subquotient of
\begin{equation*}
\delta([\nu^{j+1} \rho,\nu^{c}\rho]) \times \delta([\nu^{-i+1} \rho,\nu^{d}\rho]) \otimes  \delta([\nu^{-j} \rho,\nu^{-i-1}\rho]) \rtimes \s.
\end{equation*}
On the left-hand side appears either an irreducible representation or length two representation, and in both cases we have determined in which Jacquet modules these representations are contained.

This completes the proof of Corollary \ref{korprvi}.
\end{proof}

\section{Reducibility at zero}

The purpose of this section is to provide a complete and uniform treatment of Jacquet modules of representations of segment type in the case when the reducibility exponent $\a$ equals zero. In this case the induced representation $\rho\r\s$ reduces and we have the decomposition
$$
\rho\r\s=\tau_1\oplus\tau_{-1}
$$
into irreducible non-equivalent representations. The choice of signs $\pm$ is arbitrary, but fixed.

Fix non-negative integers $c$ and $d$ satisfying $c\leq d$, and consider the following representation:
\begin{multline*}
(\nu\rho\t \nu^{2}\rho \t\dots\t \nu^{d}\rho)\t (\nu\rho\t \nu^{2}\rho \t\dots\t \nu^{c}\rho) \t\rho\r \s \\
\cong \ \ \oplus_{i=-1}^1 \ (\nu\rho\t \nu^{2}\rho \t\dots\t \nu^{d}\rho)\t (\nu\rho\t \nu^{2}\rho \t\dots\t \nu^{c}\rho) \r\tau_i.
\end{multline*}

\begin{remark}
A direct consequence of the formula for $\mu^*$ is that no irreducible subquotient of $(\nu\rho\t \nu^{2}\rho \t\dots\t \nu^{d}\rho)\t (\nu\rho\t \nu^{2}\rho \t\dots\t \nu^{c}\rho) \r\tau_i$ can have in its Jacquet module a term of the form $\ldots \o\tau_{-i}$, for $i \in \{ 1, -1 \}$.

\end{remark}

The multiplicity of $\d([\nu\rho,\nu^{d}\rho])\t \d([\nu\rho,\nu^{c}\rho]) \o\tau_i$ in the $GL$-type Jacquet module of the above full induced representation, is one (the same holds for $s_{GL}(\d([\nu^{-c}\rho,\nu^{d}\rho])\r\s)$). We denote by $\d([\nu^{-c}\rho,\nu^{d}\rho]_{\tau_i};\s)$ the unique irreducible subquotient which has this representation in its Jacquet module (it is also a subquotient of $\d([\nu^{-c}\rho,\nu^{d}\rho])\r\s$). Such subquotient contains $\d([\nu^{-c}\rho,\nu^{d}\rho])\o\s$ in its Jacquet module.

We continue with conventions that we have introduced for positive reducibility. Therefore, $L_0( \d([\nu^{a}\rho,\nu^{b}\rho]);\s)$ denotes the usual Langlands quotient if $a\ne-b$, and $L_0( \d([\nu^{a}\rho,\nu^{b}\rho]);\s)$ $=0$  if $a=-b$. Furthermore, $L_0(\emptyset;\sigma)=0$, and $\d(\emptyset;\sigma)=\s$. If $0\not\in [a,b]$, then we take $ \d([\nu^{a}\rho,\nu^{b}\rho]_{\tau_i};\s)=0$. We also continue with symmetrization of the notation.

\begin{theorem} \label{teoremtreci} Let $c,d\in\Z$ be such that $0 \leq c\leq d$.
Then
\begin{enumerate}

\item[$(C1)$] If $c<d$, then $\delta([\nu^{-c}\rho,\nu^{d}\rho])\r\s$ is a representation of length three, and the composition series consists of two irreducible subrepresentations $\delta([\nu^{-c}\rho,\nu^{d}\rho]_{\tau_k};\s)$, $k=-1,1$, and the Langlands quotient $L(\delta([\nu^{-c}\rho,\nu^{d}\rho]);\s)$. \\ For $c=d$,  $\delta([\nu^{-c}\rho,\nu^{c}\rho])\r\s$ is a representation of length two, and the composition series consists of two irreducible subrepresentations $\delta([\nu^{-c}\rho,\nu^{c}\rho]_{\tau_k};\s)$, $k=-1,1$.

\item[$(C2)$] For $c<d$, $\delta([\nu^{-c}\rho,\nu^{d}\rho]_{\tau_k};\s)$ are square integrable. For $c=d$, the representations are tempered, but not square integrable.

\item[$(C3)$] $
s_{GL}(\delta([\nu^{-c}\rho,\nu^{d}\rho]_{\tau_k};\s))=\sum_{i=-c-1}^{-1}
\delta([\nu^{-i}\rho,\nu^{c}\rho]) \t \delta([\nu^{i+1}\rho,\nu^{d}\rho])\o\s.
$

\item[$(C4)$]  For $c<d$ holds

$s_{GL}(L(\delta([\nu^{-c}\rho,\nu^{d}\rho]);\s))
=
\sum_{i=0}^{d}
L(\delta([\nu^{-i}\rho,\nu^{c}\rho]) , \delta([\nu^{i+1}\rho,\nu^{d}\rho]))\o\s.
$

\item[$(C5)$] Suppose $c=d$. If $c > 0$ then
\begin{align*}
s_{top} (\delta([\nu^{-c}\rho,\nu^{c}\rho]_{\tau_k};\s)) & =  2\, \nu^{c}\rho\o\delta([\nu^{-c+1}\rho,\nu^{c}\rho]_{\tau_k};\s) + \\
& \qquad + \nu^{c}\rho\o L_0(\delta([\nu^{-c+1}\rho,\nu^{
c}\rho]);\s).
\end{align*}
If $c=0$ then $ s_{top}(\delta([\rho,\rho]_{\tau_k};\s))=\rho\o\s.$

\item[$(C6)$] For $c<d$
\begin{align*}
s_{top} (\delta([\nu^{-c}\rho,\nu^{d}\rho]_{\tau_k};\s)) & = \nu^{d}\rho\o\delta([\nu^{-c}\rho,\nu^{d-1}\rho]_{\tau_k};\s) + \\
& \qquad + \nu^{c}\rho\o\delta([\nu^{-c+1}\rho,\nu^{d}\rho]_{\tau_k};\s).
\end{align*}

\item[$(C7)$] For $c<d$
\begin{align*}
s_{top}(L(\delta([\nu^{-c}\rho,\nu^{d}\rho]);\s)) & = \nu^{d}\rho\o L_{0}(\delta([\nu^{-c}\rho,\nu^{d-1}\rho]);\s) + \\
& \qquad + \nu^{c}\rho\o L_{0}(\delta([\nu^{-c+1}\rho,\nu^{d}\rho]);\s).
\end{align*}

\end{enumerate}
\end{theorem}

\begin{proof}
Regarding $(C1)$, in the tempered case, length two was proved in \cite{T-seg}. Also, in the same paper $(C2)$ and $(C3)$ have been proved. On the other hand, the length three claim in $(C1)$ follows from Theorem 2.1 of \cite{Mu-CSSP}. Observe that it can be proved, in the same way as in the proof of Theorem \ref{teoremdrugi}, that the sum of all sums on the right-hand sides of $(C3)$ and $(C4)$ equals $s_{GL}(\delta([\nu^{-c}\rho,\nu^{d}\rho])\r\s)$. Thus, $(C4)$ follows. It remains to prove $(C5)$, $(C6)$ and $(C7)$.

Note that $
  s_{top}(\delta([\rho,\rho]_{\tau_i};\s))=\rho\o\s.
 $
 In the rest of the proof it is enough to consider the case $0<c+d$. The formula for $\mu^*$ now gives
\begin{equation*}
s_{top} (\delta([\nu^{-c}\rho,\nu^{d}\rho])\r\s)= \nu^{d}\rho\o\delta([\nu^{-c}\rho,\nu^{d-1}\rho])\r\s + \nu^{c}\rho\o\delta([\nu^{-c+1}\rho,\nu^{d}\rho])\r\s.
\end{equation*}

First we shall prove the theorem in the case $c=0$. The proof will be by induction on $d$.

First consider the case $d=1$.  The above formula and the previous theorem, together with symmetrization of the notation, imply
$$
s_{top} (\delta([\rho,\nu\rho])\r\s)= \nu\rho\o\tau_1+\nu\rho\o\tau_{-1} + \rho \o L_0(\nu\rho;\s).
$$
Therefore, the theorem holds in this situation ($\nu\rho\o\tau_i$ is obviously in the Jacquet module of $\delta([\rho,\nu\rho]_{\tau_i};\s)$). Fix $d>1$ and assume that the theorem holds for $d-1$. Then the inductive assumption implies the following decomposition into irreducible representations
\begin{align*}
s_{top} (\delta([\rho,\nu^{d}\rho])\r\s) & = \nu^{d}\rho\o\delta([\rho,\nu^{d-1}\rho])\r\s + \rho\o\delta([\nu\rho,\nu^{d}\rho])\r\s \\
& =\nu^{d}\rho\o\delta([\rho,\nu^{d-1}\rho]_{\tau_1};\s)
+\nu^{d}\rho\o\delta([\rho,\nu^{d-1}\rho]_{\tau_{-1}};\s) + {}
\\
& \qquad \, +\nu^{d}\rho\o L_0(\delta([\rho,\nu^{d-1}\rho]);\s) + \rho\o L_0(\delta([\nu\rho,\nu^{d}\rho]);\s).
\end{align*}

First two terms obviously belong to Jacquet modules of $\delta([\rho,\nu^{d}\rho]_{\tau_i};\s)$, $i=1,-1$, since neither  $\nu^{-(d-1)}\rho$ nor $\nu^{-d}\rho$ shows up on the cuspidal support of the discrete series (and $\nu^{d}\rho\o\delta([\rho,\nu^{d-1}\rho]_{\tau_i};\s)$ is obviously in the Jacquet module of $\delta([\rho,\nu^d\rho]_{\tau_i};\s)$). Next, considering  $\nu^{-d}\rho$, we get that the last summand is in Jacquet module of the Langlands quotient. Observe that the Langlands quotient embeds into $\delta([\nu^{-d}\rho,\rho])\r\s \h  \nu^{d}\rho \t \delta([\nu^{-d+1}\rho,\rho])\r\s$ and
this implies that the third summand is in the Jacquet module of the Langlands quotient. This proves formulas $(C5)$, $(C6)$ and $(C7)$.

Now we fix
$
c>0,
$
and assume that the theorem holds for $c-1$. We proceed with induction, similarly as in the case $c=0$. We start with the case $d=c$.
We know that in this case $\delta([\nu^{-c}\rho,\nu^{c}\rho])\r\s$ is of length two. The above formula and symmetrization of notation give
\begin{multline*}
s_{top} (\delta([\nu^{-c}\rho,\nu^{c}\rho])\r\s)= 2 \, \nu^{c}\rho\o\delta([\nu^{-c+1}\rho,\nu^{c}\rho]_{\tau_1};\s) + \\
+ 2 \, \nu^{c}\rho\o\delta([\nu^{-c+1}\rho,\nu^{c}\rho]_{\tau_{-1}};\s) + 2 \, \nu^{c}\rho\o L(\delta([\nu^{-c+1}\rho,\nu^{c}\rho]);\s).
\end{multline*}
It follows directly that $2\,\nu^{c}\rho\o\delta([\nu^{-c+1}\rho,\nu^{c}\rho]_{\tau_i};\s)$ has to belong to Jacquet module of $\delta([\nu^{-c}\rho,\nu^{c}\rho]_{\tau_i};\s)$ (recall that we know that the induced representation is of length two). Furthermore, the fact that $\delta([\nu^{-c}\rho,\nu^{c}\rho])\o\s$ is in  Jacquet module of both  $\delta([\nu^{-c}\rho,\nu^{c}\rho]_{\tau_i};\s)$ and transitivity of Jacquet modules imply that
$ \nu^c\rho\o L(\delta([\nu^{-c+1}\rho,\nu^{c}\rho]);\s)$ must be in each of the Jacquet modules. This completes the proof of the theorem in this case.

Fix $d>c$ and assume that the theorem holds for $d-1$ and $c$, and also for $c-1$ and all $d \geq c-1$. The inductive assumptions imply the following decomposition into irreducible representations
\begin{align*}    
s_{top}(\delta([&\nu^{-c}\rho,\nu^d\rho])\rtimes\sigma) \\
&=
\nu^d\rho\otimes\delta([\nu^{-c}\rho,\nu^{d-1}\rho])\rtimes\sigma+\nu^c\rho\otimes\delta([\nu^{c+1}\rho,\nu^d\rho])\rtimes\sigma
\\
&=
\nu^d\rho\otimes\delta([\nu^{-c}\rho,\nu^{d-1}\rho]_{\tau_1};\sigma) +
\nu^c\rho\otimes\delta([\nu^{-c+1}\rho,\nu^d\rho]_{\tau_1};\sigma)+{} \\
&\qquad +\nu^d\rho\otimes\delta([\nu^{-c}\rho,\nu^{d-1}\rho]_{\tau_{-1}};\sigma) +
\nu^c\rho\otimes\delta([\nu^{-c+1}\rho,\nu^d\rho]_{\tau_{-1}};\sigma)+{} \\
&\qquad +\nu^d\rho\otimes L(\delta([\nu^{-c}\rho,\nu^{d-1}\rho]);\sigma) +
\nu^c\rho\otimes L(\delta([\nu^{-c+1}\rho,\nu^d\rho]);\sigma),
\end{align*}
and the rest of the proof follows in the same way as in the proof of Theorem \ref{teoremdrugi}.
\end{proof}

Again, we have an interpretation in terms of admissible triples:

\begin{remark}
Suppose that we have $c \neq d$ and $-c \leq 0 \leq d$. Then
\begin{equation*}
\Jord(\delta([\nu^{-c}\rho,\nu^{d}\rho]_{\tau_k};\s)) =   \Jord(\s) \cup \{ (2c+1, \rho),(2d+1, \rho)\}.
\end{equation*}
Furthermore, if we denote by $\e_k$ the $\e$-function corresponding to $\delta([\nu^{-c}\rho,\nu^{d}\rho]_{\tau_k};\s)$, then $\e_k((2c+1, \rho), (2d+1, \rho )) = 1$ and $\e_k((2d+1, \rho)) = k$.
\end{remark}

Using the previous theorem and Lemma \ref{lematehn}, which also holds when reducibility point equals zero, we obtain a complete description of Jacquet modules in this case. Proof of the following corollary can be obtained in the same manner as the proof of Corollary \ref{korprvi}, details being left to the reader.

\begin{corollary}
\label{korolar-drugi}
 Let $c,d\in\Z$ be such that $c\leq d$ and $-c\leq0\leq d$.
Then
\begin{align*}
\mu^*&\left(\delta([\nu^{-c}\rho,\nu^{d}\rho]_{\tau_{i}};\s)\right) 
\\
& =  \sum_{i = -c-1}^{d-1%
}  \sum_{j=i+1}^{d}
\delta([\nu^{-i}\rho,\nu^{c}\rho])
 \times
\delta([\nu^{j+1} \rho,\nu^{d}\rho]) \otimes
\delta([\nu^{i+1} \rho,\nu^{j}\rho]_{\tau_{i}};\s) +{} 
\\
& + \mkern-20mu\sum_{-c-1\le i\le c-1%
}\ \sum_{i+1\le j\le
c}\mkern-75mu\rule[-4.5ex]{0pt}{2ex}_{i+j < -1}\mkern25mu
\delta([\nu^{-i}\rho,\nu^{c}\rho])
 \times
\delta([\nu^{j+1} \rho,\nu^{d}\rho]) \otimes
L_0(\delta([\nu^{i+1} \rho,\nu^{j}\rho]);\s) +{} 
\\[-1.5ex]
&\mkern100mu+ \mkern-10mu\sum_{i=-c-1}^{-1}
\delta([\nu^{-i}\rho,\nu^{c}\rho]) \t \delta([\nu^{i+1}\rho,\nu^{d}\rho])\o\s.
\end{align*}
For $c<d$ we have
\begin{align*}
\mu^*\big(&L(\delta([\nu^{-c}\rho,\nu^d\rho]);\sigma)\big)  
\\
&= \mu^\ast\big(L_0(\delta([\nu^{-c}\rho,\nu^d\rho]);\sigma)\big)
= \mu^\ast\big(L_{proper}(\delta([\nu^{-c}\rho,\nu^d\rho]);\sigma)\big)  
\\
&= \mkern-20mu\sum_{-c-1\le i\le d-1%
}\ \sum_{i+1\le j\le
d}\mkern-75mu\rule[-4.5ex]{0pt}{2ex}_{0\le i+j}\mkern25mu
L(\delta([\nu^{-i}\rho,\nu^c\rho]),\delta([\nu^{j+1}\rho,\nu^d\rho])\big) \otimes
L_0(\delta([\nu^{i+1}\rho,\nu^j\rho]);\sigma)) +{}  
\\[-1.5ex]
&\mkern100mu+ \mkern+1mu \sum_{i=0}^d
L(\delta([\nu^{-i}\rho,\nu^c\rho]),\delta([\nu^{i+1}\rho, \nu^d\rho])) \otimes
\sigma. 
\end{align*}
\end{corollary}

\section{Jacquet modules of strongly positive representations}

In this section we present an alternative way to determine the formula for Jacquet modules of strongly positive representations, which can be viewed as a certain generalization of representations of segment type studied in the fourth section. An analogous formula is obtained in \cite{Ma-JMSP}.

We fix self-dual cuspidal representation $\rho$ of $GL(n_{\rho}, F)$ (this defines $n_{\rho}$) and cuspidal representation $\sigma$ of $G_{n_{\sigma}}$ (this defines $n_{\s}$). We assume that $\nu^{\a} \rho \rtimes \s$ reduces for $\a > 0$ and put $\e = 1$ if $\a$ is an integer and $\e = 1/2$ otherwise. Fix
$$
n_\e<n_{\e+1}<\dots <n_\a
$$
such that $\e-1 \leq n_\e$ and $n_{i} - \a$ is an integer for $i = \e, \e+1, \ldots, \a$. Observe that in this case also $i-1 \leq n_{i}$ for all indices.

It has been proved in Theorem 3.4 of \cite{Ma-SP} that the induced representation
$$
\delta([\nu^{\e}\rho,\nu^{n_\e}\rho])\times \dots \times \delta([\nu^{\a}\rho,\nu^{n_\a}\rho]) \rtimes \sigma
$$
has a unique irreducible subrepresentation, which we will denote by $DS_{\rho;\s}(n_\a,$ $\ldots,n_\e)$. By Theorem 4.6 of \cite{Ma-SP}, this representation is strongly positive, i.e., its Jacquet module of $GL$-type contains only irreducible subquotients with all exponents being positive. Furthermore, it has been proved in \cite{Moe-Ex, Moe-T} and separately in \cite{Ma-SP} that every strongly positive discrete series which contains only twists of the representation $\rho$ in its cuspidal support is isomorphic to some $DS_{\rho;\s}(n_\a,\dots,n_\e)$.

We will denote the unique irreducible subrepresentation of
$$
\delta([\nu^{\e}\rho,\nu^{n_\e}\rho])\times \dots \times \delta([\nu^{\a}\rho,\nu^{n_\a}\rho])
$$
by $Lad_\rho(n_\a,\dots,n_\e)$. This is the ladder representation $L( \delta([\nu^{\e}\rho,\nu^{n_\e}\rho]), \dots $, $\delta([\nu^{\a}\rho,\nu^{n_\a}\rho]))$, as introduced in \cite{LM}. Uniqueness of the irreducible subrepresentation of $\delta([\nu^{\e}\rho,\nu^{n_\e}\rho])\times \dots \times \delta([\nu^{\a}\rho,\nu^{n_\a}\rho]) \rtimes \sigma$ implies
$$
DS_{\rho;\s}(n_\a,\dots,n_\e) \h Lad_\rho(n_k,\dots,n_\e)\r \s,
$$
and this implies
$$
s_{GL}(DS_{\rho;\s}(n_\a,\dots,n_\e)) \leq  s_{GL}(Lad_\rho(n_\a,\dots,n_\e)\r \s).
$$
It is not hard to see that the only term on the right-hand side of the previous inequality which has all exponents positive in cuspidal support, is $Lad_\rho(n_\a,\dots,n_e)\o \s$.
Therefore, $s_{GL}(DS_{\rho;\s}(n_\a,\dots,n_\e)) \leq  Lad_\rho(n_\a,\dots,n_\e)\o \s$, which implies
$$
s_{GL}(DS_{\rho;\s}(n_\a,\dots,n_\e)) =  Lad_\rho(n_\a,\dots,n_\e)\o \s.
$$
Using Lemma 3.5 of \cite{Ma-JMSP} we see that this Jacquet module uniquely characterizes the strongly positive representation.

Using the formula for Jacquet modules of ladder representations from \cite{KL}, we deduce
\begin{multline*}
(m^*\o id)(s_{GL}(DS_{\rho;\s}(n_\a,\dots,n_\e))) = m^*( Lad_\rho(n_\a,\dots,n_\e))\o \s= \\
\sum_{c_\e < \dots < c_\a,
\atop
i-1 \leq c_{i} \leq n_{i}}
L(\d([\nu^{c_\e+1}\rho,\nu^{n_\e}\rho]), \ldots, \d([\nu^{c_\a+1}\rho,\nu^{n_\a}\rho])) \o Lad_\rho(c_\a,\dots,c_\e)\o \s,
\end{multline*}
which directly gives, using the above characterization of strongly positive representations by $GL$-type Jacquet modules,
\begin{multline*}
\mu^*(DS_{\rho;\s}(n_\a,\dots,n_\e)) = \\
\sum_{c_\e < \dots < c_\a,
\atop
i-1 \leq c_{i} \leq n_{i}}
L(\d([\nu^{c_\e+1}\rho,\nu^{n_\e}\rho]), \ldots, \d([\nu^{c_\a+1}\rho,\nu^{n_\a}\rho]))  \o DS_{\rho;\s}(c_\a,\dots,c_\e).
\end{multline*}

Now we shall try to give Jordan blocks interpretation of the above formula. Fix $\rho$ and $\s$ as above, and a sequence of integers
$$
k_{\lceil \a \rceil}>\dots >k_1\geq 0,
$$
where $\lceil \a \rceil$ denotes the smallest integer which is not smaller than $\a$. Integers $k_{1}, \ldots, k_{\lceil \a \rceil}$ are taken to be odd if $\a$ is integral. Otherwise, we take them to be even.

Denote $\Jord_{(\rho; k_{\lceil \a \rceil},\dots ,k_1)}=\{(\rho, k_{\lceil \a \rceil}),\dots ,(\rho,k_1)\}$, where we drop $(\rho,0)$ if it shows up on the right-hand side. As we have seen before, the induced representation
$$
\delta([\nu^{\lceil \a + \frac{1}{2} \rceil - \a}\rho,\nu^{(k_1-1)/2}\rho])\times \dots \times \delta([\nu^{\a}\rho,\nu^{(k_{\lceil \a \rceil}-1)/2}\rho]) \rtimes \sigma
$$
contains a unique irreducible subrepresentation, which is denoted by $\lambda_{\{(\rho, k_{\lceil \a \rceil}),\dots ,(\rho,k_1)\},\e_+,\s}$.
It is a discrete series reresentation and, by \cite{Moe-Ex}, it is attached to an admissible triple. But for such admissible triples we have an alternated partially defined function and there is at most one such function, which is already determined by Jordan blocks and the partial cuspidal support $\s $.
Now the above formula in this notation becomes
\begin{multline*}
\mu^*(\lambda_{\{(\rho, k_{\lceil \a \rceil}),\dots ,(\rho,k_1)\},\e_+,\s}) = \\
\sum_{l_1< \dots< l_{\lceil \a \rceil}
\atop
2(\lceil \a \rceil - \a + i)-1 \leq l_{i} \leq k_{i} + 1 }
L(\d([\nu^{(l_1+1)/2}\rho,\nu^{(k_1-1)/2}\rho]), \ldots, \d([\nu^{(l_{\lceil \a \rceil}+1)/2}\rho,\nu^{(k_{\lceil \a \rceil}-1)/2}\rho])) \o{} \\
 \o \lambda_{\{(\rho, l_{\lceil \a \rceil}),\dots ,(\rho,l_1)\},\e_+,\s},
\end{multline*}
where $k_{i}- l_{i}$ are integers for all $i$.

\section{Top Jacquet modules}

This section is devoted to determination of top Jacquet modules of general discrete series of classical groups.

Let us denote a discrete series representation by $\pi$, corresponding to an admissible triple $(\Jord(\pi),\ep,\pc)$.

The facts which we collect in the following lemma are well known (see \cite{Moe-Ex} and \cite{Moe-T}).

\begin{lemma} Suppose that $\tau\o\varphi$ is an irreducible representation contained in $s_{top}(\pi)$ and that $\rho$ is an irreducible self-dual representation of a general linear group. Let $\tau = \nu^{e(\tau)} \tau_{u}$, with $\tau_{u}$ unitarizable. Then
\begin{itemize}

\item[] $\tau_u$ is self-dual, $e(\tau)\in(1/2)\Z$ and $e(\tau)>0$;

\item[] $(\tau_u,2e(\tau)+1)\in \Jord(\pi)$;

\item[]  if $e(\tau) = \frac{1}{2}$ then $\ep((\rho,2))=1$;

\item[]  if $(\tau_u,2e(\tau)-1)\in \Jord(\pi)$, then $\ep((\rho,2e(\tau)-1),(\rho,2e(\tau)+1))=1$.
\end{itemize}
\end{lemma}

First we shall consider the  situation when
$$
2\in \Jord\nolimits_{\rho}(\pi) \text{ \ and \ }\ep((\rho,2))=1.
$$
In  this case $\pi^{(\rho,2\downarrow\emptyset)}$ or $\pi^{(\rho,2\downarrow0)}$} will denote the irreducible square integrable representation determined by an admissible triple
$$
(\Jord\nolimits_{\rho}(\pi) \backslash \{(\rho,2)\},\ep',\pc),
$$
where $\ep'$ denotes the partially defined function which one gets by restriction of $\ep$ to $\Jord_{\rho}(\pi) \backslash \{(\rho,2)\}$.

\begin{lemma} Let $\rho$ be an irreducible self-dual representation of a general linear group. Suppose that $2\in \Jord_{\rho}(\pi)$, $\ep((\rho,2))=1$ and $\nu^{1/2}\rho\o\varphi\leq s_{top}(\pi)$
for some irreducible $\varphi$. Then $\varphi\cong\pi^{\d(\rho,2\downarrow 0)}$
and the multiplicity of $\nu^{1/2}\rho\o\varphi$ in  $s_{top}(\pi)$ is one.
\end{lemma}

\begin{proof}
Since $\ep((\rho,2))=1$, we have
$$
\pi\h \nu^{1/2}\rho\r \pi^{(\rho,2\downarrow0)}
$$
(see Lemma 9.1 of \cite{T-temp}, or \cite{Jn-temp}). Thus
$$
\nu^{1/2}\rho\o\varphi\leq \mu^*(\nu^{1/2}\rho\r\pi^{(\rho,2\downarrow0)}),
$$
which directly implies (by the formula for $\mu^*$)
$$
\nu^{1/2}\rho\o\varphi\leq  (\nu^{1/2}\rho\o1)\r \mu^*(\pi^{(\rho,2\downarrow0)})+(1\o\nu^{1/2}\rho)\r\mu^*(\pi^{(\rho,2\downarrow0)}).
$$
We have two possibilities. The first one is $\nu^{1/2}\rho\o\varphi\leq  (\nu^{1/2}\rho\o1)\r \mu^*(\pi^{(\rho,2\downarrow0)})$, which implies $\varphi\cong \pi^{(\rho,2\downarrow0)}$, and the second one is $\nu^{1/2}\rho\o\varphi\leq  (1\o\nu^{1/2}\rho)\r\mu^*(\pi^{(\rho,2\downarrow0)})$, which directly implies that $(\rho,2)$ is in the Jordan block of $\pi^{(\rho,2\downarrow0)}$, a contradiction. Consequently, $\varphi\cong \pi^{(\rho,2\downarrow0)}$.

Furthermore, the assumption $2 \nu^{1/2}\rho\o \pi^{(\rho,2\downarrow0)} \leq s_{top}(\pi)$ would give $2\nu^{1/2}\rho\o \pi^{(\rho,2\downarrow0)} \leq \nu^{1/2}\rho\o \pi^{(\rho,2\downarrow0)}$, which is impossible. This  completes the proof.
\end{proof}

Now  we shall consider the situation when
 $$
 \text{$a\geq 3$, $a\in \Jord\nolimits_{\rho}(\pi)$ and $a-2\not\in \Jord\nolimits_{\rho}(\pi).$}
 $$
Then $\pi^{(\rho,a\downarrow a-2)}$ will denote the irreducible square integrable representation determined by the admissible triple given as follows: the Jordan blocks are obtained by replacing, in $\Jord_{\rho}(\pi)$, the representation $(\rho,a)$ by $(\rho,a-2)$  and keeping all other representations unchanged. The new partially defined function is obtained from the old one by replacing everywhere $(\rho,a)$ by $(\rho,a-2)$, while the partial cuspidal support remains unchanged.

\begin{lemma} Let $\rho$ be an irreducible self-dual representation of a general linear group. Suppose that $a\geq 3$, $a\in \Jord_{\rho}(\pi)$ and $a-2\not\in \Jord_{\rho}(\pi)$ and
$$
\nu^{(a-1)/2}\rho\o\varphi\leq s_{top}(\pi)
$$
for some irreducible $\varphi$. Then $
\varphi\cong\pi^{(\rho,a\downarrow a-2)}$, and the multiplicity of $\nu^{(a-1)/2}\rho\o\varphi$ in  $s_{top}(\pi)$ is one.
\end{lemma}

\begin{proof}
Lemma 8.1 of \cite{T-temp} gives $
\pi\h \nu^{(a-1)/2}\rho\r \pi^{(\rho,a\downarrow a-2)}$. Now the rest of the proof runs in the same way as in the proof of the previous lemma.
\end{proof}

At the end, we shall consider the situation
$$
\text{$a, a-2\in \Jord\nolimits_{\rho}(\pi)$ and $\ep((\rho,a-2),(\rho,a))=1$.}
$$

Here we shall need the parametrization of tempered duals. We choose to work with the one from \cite{Jn-temp}.

Jantzen parameters are very similar to parameters of the square integrable representations, but here parameters are quadruples, where the additional parameter is the multiplicity function on Jordan blocks. However, we can interpret these parameters as triples, by interpreting Jordan blocks $\Jord(\tau)$ as multisets. When we consider the set determined by $\Jord(\tau)$ (this is the case when one considers the partially defined function attached to $\tau$ in \cite{Jn-temp}), then it will be determined by $|\Jord(\tau)|$.

Let us denote by $\pi_0$ the irreducible discrete series determined by admissible triple given in the following way: the Jordan blocks are obtained by removing $(\rho,a)$ and $(\rho,a-2)$ in $\Jord_{\rho}(\pi)$, and the new partially defined function is obtained from the old one by restriction, while the partial cuspidal support remains unchanged.

Consider now two inequivalent tempered irreducible subrepresentations of
$$
\d([\nu^{-\frac{a-3}{2}}\rho, \nu^{\frac{a-3}{2}}\rho])\r\pi_0=\tau_1+\tau_{-1}.
$$
For precisely one $i_{0} \in \{ 1, -1 \}$, we have
$$
\pi\h \nu^{(a-1)/2}\rho \r\tau_{i_0}.
$$
Now we shall discuss the Jantzen parameters of  representations $\tau_i$.
The partial cuspidal supports of both $\tau_i$'s are  the same as of $\pi_0$ (and $\pi$).
Furthermore, one gets Jordan blocks of both $\tau_i$'s  by adding twice $(\rho,a-2)$ to $\Jord(\pi_0)$. In other words, one gets $\Jord(\tau_i)$ from $\Jord(\pi)$ by replacing
$(\rho,a)$ by $(\rho,a-2)$ (not forgetting that we have now the multiplicity two of $(\rho,a-2)$).

One has two possibilities for the partially defined functions corresponding to representations $\tau_i$. Let us denote by $\e'$ the partially defined function on $|\Jord(\tau)|$ which one gets from $\ep$ replacing $(\rho,a)$ by $(\rho,a-2)$ everywhere in the definition of $\ep$. Now for precisely one of the $\tau_i$'s as above, the partially defined function of $\tau_i$ is equal to $\e'$.
We denote $\tau_i$ corresponding to this partially defined function by $\pi^{(\rho,a \downarrow a-2)}$.

Now Corollary 3.2.3 of \cite{Jn-temp} implies $\pi \h \nu^{(a-1)/2}\rho \r \pi^{(\rho,a \downarrow a-2)}$.

\begin{lemma} Let $\rho$ be an irreducible self-dual representation of a general linear group. Suppose that $a\geq 3$, $a, a-2\in \Jord_{\rho}(\pi)$, $\ep((\rho,a-2),(\rho,a))=1$ and
$$
\nu^{(a-1)/2}\rho\o\varphi\leq s_{top}(\pi)
$$
for some irreducible $\varphi$. Then $\varphi\cong \pi^{(\rho,a \downarrow a-2)}.$
The multiplicity of $\nu^{(a-1)/2}\rho\o\varphi$ in  $s_{top}(\pi)$ is one.
\end{lemma}

\begin{proof}We know that
$$
\nu^{(a-1)/2}\rho\o\varphi\leq \mu^*(\pi)\leq \mu^*(\nu^{(a-1)/2}\r\pi^{(\rho,a \downarrow a-2)}),
$$
which directly implies
$$
\nu^{(a-1)/2}\rho\o\varphi\leq  (\nu^{(a-1)/2}\rho\o1)\r \mu^*(\pi^{(\rho,a \downarrow a-2)})+(1\o\nu^{(a-1)/2}\rho) \r\mu^*(\pi^{(\rho,a \downarrow a-2)}).
$$
Again there are two possibilities. The first one is $\nu^{(a-1)/2}\rho\o\varphi\leq  (\nu^{(a-1)/2}\rho\o1)\r \mu^*(\pi^{(\rho,a \downarrow a-2)})$, which implies $\varphi\cong \pi^{(\rho,a \downarrow a-2)}$.

The remaining possibility is $\nu^{(a-1)/2}\rho\o\varphi\leq  (1\o\nu^{(a-1)/2}\rho)\r\mu^*(\pi^{(\rho,a \downarrow a-2)})$. This implies
$$
(\nu^{(a-1)/2}\rho\o\varphi)\leq (1\o\nu^{(a-1)/2}\rho)\t M^*(\d([\nu^{-\frac{a-3}{2}}\rho, \nu^{\frac{a-3}{2}}\rho]))\r\mu^*(\pi_0).
 $$
The formula for $M^*(\d([\nu^{-\frac{a-3}{2}}\rho, \nu^{\frac{a-3}{2}}\rho]))$ gives
$$
\nu^{(a-1)/2}\rho\o \varphi'\leq \mu^*(\pi_0)
$$
for some $\varphi'$, which further implies that $(\rho,a)$ is in the Jordan block of $\pi_0$, which is not the case. Thus, we got a contradiction. Consequently, $\varphi\cong \pi^{(\rho,a \downarrow a-2)}$. The assumption $2\,\nu^{(a-1)/2}\rho\o\varphi\leq s_{top}(\pi)$ would imply that $$\nu^{(a-1)/2}\rho\o\d([\nu^{-\frac{a-3}{2}}\rho, \nu^{\frac{a-3}{2}}\rho])\r \pi'$$ is not a multiplicity one representation, which is impossible since $\d([\nu^{-\frac{a-3}{2}}\rho$, $\nu^{\frac{a-3}{2}}\rho])\r \pi'$ is a multiplicity one representation.

This completes the proof.
\end{proof}

From the above four lemmas we obtain the following

\begin{theorem} Let $\pi$ be an irreducible square integrable representation of a classical group. Then
$$
s_{top}(\pi)=
\sum \nu^{(a-1)/2}\rho\o \pi^{(\rho,a \downarrow a-2)},
$$
where the sum runs over all $(\rho,a)\in \Jord(\pi)$ which satisfy the following two conditions:
\begin{align*}
a-2\in \Jord\nolimits_{\rho}(\pi) & \Rightarrow \ep((\rho,a-2),(\rho,a))=1; \\
a=2& \Rightarrow \ep((\rho,2))=1.
\end{align*}
\end{theorem}

\bibliographystyle{siam}
\bibliography{Literatura}

\end{document}